\documentclass[12pt]{amsart}
\usepackage{amsmath,mathtools,amsthm,amsfonts,bigints,amssymb, mathrsfs, tikz-cd, xcolor, float}
\usepackage{tikz}
\usetikzlibrary{patterns}
\usepackage{pgf,tikz}

\usepackage{pgfplots}
\usepgfplotslibrary{fillbetween}

\usetikzlibrary{arrows}

\newtheorem{theorem}{Theorem}[section]
\newtheorem{lemma}[theorem]{Lemma}

\newtheorem{corollary}[theorem]{Corollary}

\theoremstyle{definition}
\newtheorem{definition}[theorem]{Definition}

\newtheorem{remark}[theorem]{Remark}

\newcommand{\R}{\mathbb R}%
\newcommand{\M}{\mathcal M}%
\newcommand{\C}{\mathbb C}%
\newcommand{\N}{\mathbb N}%
\newcommand{\PP}{\mathcal P}%
\newcommand{\Ss}{\mathcal S}%
\newcommand{\X}{\mathbb X}%

\newcommand{\B}{\mathscr B}%
\numberwithin{equation}{section}

\usepackage[colorlinks=true,linkcolor=purple
,citecolor=blue]{hyperref}
\makeatletter

\renewcommand\subsubsection{\@secnumfont}{\bfseries}%
\renewcommand\subsubsection{\@startsection{subsubsection}{3}
  \z@{.5\linespacing\@plus.7\linespacing}{-.5em}%
  {\normalfont\bfseries}}

  \makeatother
\makeatletter
\@namedef{subjclassname@2020}{%
  \textup{2020} Mathematics Subject Classification}
\makeatother

\begin{document}

\title[Boundary behaviour of eigenfunctions and superharmonic functions]{Boundary behaviour of eigenfunctions and superharmonic functions on Harmonic manifolds of purely exponential volume growth}

\author[Utsav Dewan]{Utsav Dewan}
\address{Department of Mathematics, Indian Institute of Technology Bombay, Powai, Mumbai-400076, India}
\email{utsav@math.iitb.ac.in}



\begin{abstract}
On $\mathbb{X}$, a non-positively curved harmonic manifold of purely exponential volume growth, of dimension $n \ge 3$, we study certain quantitative aspects of the boundary behaviour of eigenfunctions and superharmonic functions. More precisely, we first consider $u$, complex-valued eigenfunctions lying outside the $L^2$-spectrum of $\Delta$ and prove that/obtain:
\begin{enumerate}
\item specific weighted non-tangential limits of $u$ exist a.e. on the boundary and agree with the Radon-Nikodym derivative of the absolutely continuous component of its boundary measure, with respect to the visibility measure;
\item Hausdorff dimension and Hausdorff measure estimates of the boundary exceptional sets where radial limits of $u$ blow-up;
\item sharpness of the size estimates obtained in (2).
\end{enumerate}
Then we focus on positive superharmonic functions $f$ and prove that/obtain:
\begin{itemize}
\item[(4)] $f$ may not have non-tangential limits, by constructing a counter-example whose non-tangential limits blow up at all boundary points;
\item[(5)] if, however, the Riesz measure of $f$ is absolutely continuous with respect to the volume measure, with density belonging to specific weighted $L^p$ spaces for $p>n/2$, then non-tangential limits of $f$ exist a.e. on the boundary and agree with the Radon-Nikodym derivative of the absolutely continuous component of the boundary measure of its greatest harmonic minorant, with respect to the visibility measure;  
\item[(6)] the integrability range of $p>n/2$ in (5) is sharp, by means of counter-examples;
\item[(7)] if in addition to (5), the density has extra decay near infinity and the asymptotic spherical averages of $f$ also decay, then $f$ has tangential limits on the boundary with the exceptional set having prescribed Hausdorff dimension.
\end{itemize}
The results (2)-(7) are new even for the homogeneous setting of rank one Riemannian symmetric spaces of non-compact type and Damek-Ricci spaces. Our arguments are based on potential theory adapted to the intrinsic Gromov hyperbolic geometry of $\mathbb{X}$.
\end{abstract}

\maketitle
\tableofcontents

\section{Introduction}
\label{sec1}
The connection between the geometry of a complete Riemannian manifold and the existence of non-constant positive harmonic functions thereon, has long been a focal point in geometric analysis. On one hand, in the regime of non-negative Ricci curvature, Yau \cite{Yau} proved that every positive harmonic function must be constant. On the other hand, on Hadamard manifolds of pinched negative sectional curvature, Anderson-Schoen \cite{AS} exhibited abundance of non-constant positive harmonic functions by establishing a natural homeomorphism between the Martin boundary and the boundary at infinity. 

\medskip 

In \cite{AS}, the authors also proved the existence of non-tangential limits of positive harmonic functions, almost everywhere on the boundary with respect to the harmonic measures, obtaining a remarkable generalization of the classical Fatou's theorem \cite{Fatou}. This study was furthered by Ancona for more general elliptic operators in \cite{Ancona}.

\medskip

Due to the negative curvature, the Laplace-Beltrami operator $\Delta$ has a spectral gap and hence harmonic functions lie outside the $L^2$-spectrum of $\Delta$. Thus in light of the above results in \cite{AS, Ancona}, one is naturally led to ask about the boundary behaviour of other eigenfunctions of $\Delta$, lying outside the $L^2$-spectrum:
\begin{itemize}
\item[($\mathcal{Q}1$)]  {\it What is the admissible rate of convergence of non-tangential boundary limits for other eigenfunctions, on a full harmonic measure subset of the boundary at infinity?}
\item[($\mathcal{Q}2$)] {\it How does a positive eigenfunction behave along radial geodesic rays, on the complement of this full measure subset of the boundary? More precisely, how quickly can a positive eigenfunction grow or how large can the exceptional set (in the boundary) be where the positive eigenfunction  blows up faster than a prescribed rate?}
\end{itemize}

The above quantitative questions largely remain unanswered in the generality of Hadamard manifolds of pinched negative curvature. In the simplest case of the real hyperbolic ball, however, the problem ($\mathcal{Q}1$) is well understood \cite{Minemura}. In fact, studies pertaining to the problem ($\mathcal{Q}1$) extend to the class of homogeneous Hadamard manifolds of pinched negative curvature, known as rank one Riemannian symmetric spaces of non-compact type \cite{Michelson, Sjogren, Schlichtkrull}.

\medskip

Besides the delicate homogeneous structure of having a large isometry group acting transitively on the unit tangent bundle, the rank one Riemannian symmetric spaces of non-compact type $\X$ enjoy two interesting fine and coarse geometric properties. Firstly, for any point $x \in \X$, there exists a non-constant harmonic function on a punctured neighborhood of $x$ which is radial around $x$, i.e., only depends on the geodesic distance from $x$. Such manifolds are known as {\it harmonic manifolds}. On the other hand, they also satisfy the coarse geometric property of {\it purely exponential volume growth}: there exists a constant $h>0$ called the {\em logarithmic volume entropy}, such that the volume of geodesic balls $B(x,R)$ with center $x \in \X$ and radius $R>1$ satisfy:
\begin{equation} \label{volume_growth}
vol\left(B(x,R)\right) \asymp e^{hR}\:.
\end{equation}

It is noteworthy to mention that the class of harmonic manifolds of purely exponential volume growth is much larger than the class of rank one Riemannian symmetric spaces of non-compact type, as it also includes a general class of non-symmetric, non-compact harmonic manifolds. Introduced by Damek-Ricci in \cite{DR}, these are non-unimodular, solvable extensions of Heisenberg type groups $N$, obtained by letting $A=\R^+$ act on $N$ by non-isotropic dilations, and have come to be known as harmonic $NA$ groups or Damek-Ricci spaces. The class of rank one Riemannian Symmetric spaces of noncompact type (barring the degenerate case of the real hyperbolic spaces) is contained in (and in fact accounts for a very small subclass of) the more general class of Damek-Ricci spaces \cite[p. 643]{ADY}. In addition to being non-symmetric, Damek-Ricci spaces are also non-positively curved and are the only known examples of non-compact, non-flat harmonic manifolds. In the setting of Damek-Ricci spaces, boundary behaviour of eigenfunctions lying outside the $L^2$-spectrum of $\Delta$, has recently been studied in \cite{Ray, ProcMathSci}.

\medskip

The harmonic analysis of rank one Riemannian symmetric spaces of non-compact type and Damek-Ricci spaces depends intrinsically on their Lie algebraic structures. Moreover, their Martin boundary is naturally identified with the one-point compactification of the homogeneous group $N$. Thus we are prompted to ask whether similar quantitative behaviours persist in general non-positively curved harmonic manifolds $\X$ of purely exponential volume growth, defined solely by coarse and fine geometric structures, with their boundary at infinity $\partial \X$ formed through an asymptotic dynamical process?

\medskip

We now commence our exploration on harmonic manifolds. For unexplained notations and terminologies, we refer the reader to Section \ref{sec2}. All manifolds in our discussion are assumed to be complete and simply-connected. Let $\X$ be a non-positively curved harmonic manifold of purely exponential volume growth with mean curvature of horospheres $h>0$ with dimension $\ge 3$. In this case, the mean curvature of horospheres coincides with the logarithmic volume entropy defined in (\ref{volume_growth}). We fix an origin $o \in \X$ and let $\partial \X$ be the boundary at infinity. Now for a $\xi \in \partial \X$, let $\gamma_\xi$ denote the unit-speed geodesic ray such that $\gamma_\xi(0)=o,\:\gamma_\xi(+\infty)=\xi$ and $d(o,\gamma_\xi(t))=t$ for all $t \in (0,+\infty)$. Then the Poisson kernel of $\X$ is given by,
 \begin{equation*} 
 P(x,\xi) = e^{-hB_{\xi}(x)} \:,\text{ for all } x \in \X,\: \xi \in \partial \X,
 \end{equation*}
 where $B_{\xi}(x)$ is the Busemann function, defined by
 \begin{equation} \label{busemann} 
 B_{\xi}(x) = \displaystyle\lim_{t \to \infty} \left(d\left(x,\gamma_\xi(t)\right) - d\left(o,\gamma_\xi(t)\right)\right) \:.
 \end{equation} 

The Martin representation formula \cite[Corollary 5.13]{KL} asserts that the positive harmonic functions on $\X$ are given by Poisson  integrals of Radon measures $\mu$ on $\partial \X$, that is
\begin{equation} \label{martin_rep}
\PP[\mu](x):= \int_{\partial \X} P(x,\xi)\:d\mu(\xi)\:,\: x \in \X\:.
\end{equation}
More generally, for a complex measure $\mu$ on $\partial \X$, $u:=\PP[\mu]$, the Poisson integral of $\mu$ is a complex-valued harmonic function on $\X$. Now a speciality of our present setting is that we can obtain more general eigenfunctions of $\Delta$, simply by taking suitable powers of the Poisson kernel. More precisely, setting $\rho=h/2$, for $\lambda \in \C$, by considering the horospherical part of $\Delta$, we note that the function $x \mapsto e^{(i\lambda -\rho)B_\xi(x)},\:\xi \in \partial \X,$ is an eigenfunction of $\Delta$ with eigenvalue $-(\lambda^2+\rho^2)$ \cite[Proposition 3.3]{BKP}. 

\medskip

Now it is well-known \cite[Corollary 5.2]{PS} that, the supremum of the spectrum and of the essential spectrum agree and equal $-\rho^2$. Thus for outside the spectrum, i.e. for $\lambda=i\beta$, with $\beta>0$, it becomes natural to study the boundary behaviour of the generalized Poisson integrals of complex measures $\mu$:
\begin{equation} \label{generalized_poisson_integral}
\PP_{i\beta}[\mu](x):= \int_{\partial \X} P_{i\beta}(x,\xi)\:d\mu(\xi)\:,\: x \in \X\:,
\end{equation}
where
\begin{equation} \label{generalized_poisson_kernel}
P_{i\beta}(x,\xi):=C_\beta \left(P(x,\xi)\right)^{\frac{\beta}{h}+\frac{1}{2}}= C_\beta \: e^{-(\beta +\rho)B_\xi(x)}\:,
\end{equation}
$C_\beta >0$ is a normalizing constant so that
\begin{equation} \label{approx_identity}
e^{-(\beta - \rho)d(o,x)}\int_{\partial \X} P_{i\beta}(x,\xi) \:d\lambda_o(\xi)= 1\:,\: x \in \X\:,
\end{equation}
where $\lambda_o$ is the visibility measure on $\partial \X$ with respect to $o$.

\medskip

This brings us to the notion of non-tangential limits, which is defined in terms of Gromov products:
\begin{definition} \label{non-tangential_limit_defn}
\begin{itemize}
\item[(i)] For $\xi \in \partial \X$ and $\alpha >1$, {\it the non-tangential cone based at $\xi$ with aperture $\alpha$} is defined to be,  
\begin{equation*}
\Gamma_{\alpha}(\xi):= \left\{x \in \X \mid e^{-\left(x |\xi\right)_o}< \alpha e^{-d(o,x)}\right\}\:.
\end{equation*}
\item[(ii)] A function $f$ on $\X$ is said to have non-tangential limit $L \in \C$ at $\xi \in \partial \X$ if
\begin{equation*}
\displaystyle\lim_{\substack{x \to \xi \\ x \in \Gamma_\alpha(\xi)}} f(x)=L\:,\:\:\forall \alpha >1\:.
\end{equation*}
\end{itemize}
\end{definition}

Another important aspect of our study, is the metric geometry at infinity. When the sectional curvature $K_\X \le -1$, there is a natural metric on the boundary at infinity $\partial \X$ defined in terms of the Gromov product,
\begin{equation} \label{visual_metric}
\nu(\xi,\eta):=e^{-(\xi|\eta)_o}\:,\:\:\xi,\eta \in \partial \X\:,
\end{equation}
called the {\em visual metric}. But in the generality of non-positive variable sectional curvature, $\nu$ only defines a quasi-metric. Henceforth, in all our results, the Hausdorff dimensions or Hausdorff outer measures are with respect to $\nu$. The other crucial aspect is the growth of the visibility measure of visual balls
\begin{equation*}
\mathscr{B}(\xi,r):=\left\{\eta \in \partial \X \mid \nu(\xi,\eta)<r\right\}\:, \:\:\xi \in \partial \X\:,\:r \in (0,1)\:.
\end{equation*} 
This prompts us to look at the following definition:
\begin{definition} \label{doubling_defn}
The triplet $\left(\partial \X,\nu,\lambda_o\right)$ is called
 {\it $h$-doubling} if there exists $C>0$ such that for all visual balls
\begin{equation*}
\lambda_o\left(\mathscr{B}(\xi,2r)\right) \le C\:2^h \lambda_o\left(\mathscr{B}(\xi,r)\right)\:.
\end{equation*}
\end{definition}

\begin{remark} \label{boundary_doubling_remark}
In the generality of harmonic manifolds of purely exponential volume growth, only the following bound:
\begin{equation*} 
\lambda_o\left(\B(\xi,r)\right) \le C r^h\:,\:\:r \in (0,1)\:,
\end{equation*}  
is well-known (see \cite[Lemma 6.3]{BKP}). However, the notion of $h$-doubling is particularly relevant as it is satisfied for all the known cases, i.e. rank one Riemannian symmetric spaces of non-compact type and Damek-Ricci spaces.
\end{remark}

We are now ready to present our first result:
\begin{theorem} \label{Fatou_thm}
Let $\X$ be a non-positively curved harmonic manifold of purely exponential volume growth with mean curvature of horospheres $h>0$, with $\left(\partial \X,\nu,\lambda_o\right)$ being $h$-doubling	 and set $\rho=h/2$. Then for complex measures $\mu$ on $\partial \X$ and $\beta >0$, the non-tangential limit of  $e^{-(\beta - \rho)d(o,\cdot)}\PP_{i\beta}[\mu]$ exists at $\lambda_o$-a.e. point on $\partial \X$ and equals the Radon-Nikodym derivative $\frac{d\mu_{ac}}{d\lambda_o}$, where $\mu_{ac}$ is the absolutely continuous component of $\mu$ with respect to $\lambda_o$.
\end{theorem}

We now turn our attention to the problem $(\mathcal{Q}2)$, i.e. the size estimates of the blow-up set for radial limits of eigenfunctions, on the boundary $\partial \X$. In order to do so, we do a quick computation. For $\beta>0$, a complex measure $\mu$ on $\partial \X$, $\xi \in \partial \X$ and any $t \in (0,+\infty)$, by (\ref{generalized_poisson_integral}), (\ref{generalized_poisson_kernel}) and the triangle inequality, we have
\begin{equation} \label{bounding_generalized_poisson}
|e^{-(\beta - \rho)t}\PP_{i\beta}[\mu](\gamma_\xi(t))| =C_\beta e^{-(\beta - \rho)t}\left|\int_{\partial \X} P_{i\beta}(\gamma_\xi(t),\eta)\:d\mu(\eta)\right| \le \left(C_\beta |\mu|(\partial \X)\right) e^{ht} \:,
\end{equation}
where $|\mu|(\partial \X)$ is the total variation of $\mu$.

\medskip

Then (\ref{bounding_generalized_poisson}) motivates us to consider for $\beta >0$ and $\alpha \in [0,h]$,  the following sets
\begin{equation} \label{E_alpha}
E_\alpha\left(\PP_{i\beta}[\mu]\right) := \left\{\xi \in \partial \X : \displaystyle\limsup_{t \to +\infty} e^{-(\alpha + \beta -\rho) t} \left|\PP_{i\beta}[\mu]\left(\gamma_{\xi}(t)\right)\right| > 0\right\}\:,
\end{equation}
and
\begin{equation} \label{E_alpha_inf}
E^\infty_\alpha\left(\PP_{i\beta}[\mu]\right) := \left\{\xi \in \partial \X : \displaystyle\limsup_{t \to +\infty} e^{-(\alpha + \beta -\rho) t} \left|\PP_{i\beta}[\mu]\left(\gamma_{\xi}(t)\right)\right| =+\infty\right\} \:.
\end{equation}

\medskip

Removing the $h$-doubling constraint of Theorem \ref{Fatou_thm}, we now estimate the size of these exceptional sets:
\begin{theorem} \label{exceptional_thm}
Let $\X$ be a non-positively curved harmonic manifold of purely exponential volume growth with mean curvature of horospheres $h>0$ and set $\rho=h/2$. Assume that $\beta>0,\:\alpha \in [0,h]$ and $\mu$ is a complex measure on $\partial \X$. Then 
\begin{equation*}
dim_{\mathcal{H}} E_\alpha\left(\PP_{i\beta}[\mu]\right) \le h-\alpha \:, \text{ and }
\mathcal{H}^{h-\alpha}\left(E^\infty_\alpha(\PP_{i\beta}[\mu])\right) = 0 \:.
\end{equation*}
\end{theorem}

We are now naturally led to ask whether the size estimates obtained in Theorem \ref{exceptional_thm} are sharp. More precisely,
\begin{itemize}
\item[$(\mathcal{Q}3)$] {\it Given $\beta>0,\:\alpha \in [0,h)$ and any $E \subset \partial \X$ with $\mathcal{H}^{h-\alpha}(E)=0$, does there exist a Radon measure $\mu$ on $\partial \X$, such that $E \subset E^\infty_\alpha(\PP_{i\beta}[\mu])$\:?}
\item[$(\mathcal{Q}4)$] {\it Given $\beta>0,\:\alpha \in [0,h)$, does there exist a Radon measure $\mu$ on $\partial \X$, such that} $$dim_{\mathcal{H}}\left(E_\alpha\left(\PP_{i\beta}[\mu]\right)\right) = h-\alpha\:?$$ 
\end{itemize}

Our next result provides affirmative answers to both the problems $(\mathcal{Q}3)$ and $(\mathcal{Q}4)$:
\begin{theorem} \label{exceptional_sharp_thm}
Let $\X,\:h,\:\rho$ be as in Theorem \ref{exceptional_thm} and  $\beta>0\:,\alpha \in [0,h)$.
\begin{itemize}
\item[(i)] Given any $E \subset \partial \X$ with $\mathcal{H}^{h-\alpha}(E)=0$, there exists a Radon measure $\mu$ on $\partial \X$, such that $E \subset E^\infty_\alpha(\PP_{i\beta}[\mu])\:.$
\item[(ii)] There exists a Radon measure $\mu$ on $\partial \X$, such that $dim_{\mathcal{H}}\left(E_\alpha\left(\PP_{i\beta}[\mu]\right)\right) = h-\alpha\:.$
\end{itemize}
\end{theorem}

\begin{remark} \label{exceptional_thms_rmk}
 Theorems \ref{exceptional_thm} and \ref{exceptional_sharp_thm} on boundary exceptional sets for eigenfunctions, are new even for the simplest case of the real hyperbolic space.
\end{remark}

The fractal structure of the boundary at infinity is a classic point of interest in geometry. For Hadamard manifolds of pinched negative curvature, the Hausdorff dimension of harmonic measures on $\partial \X$ has been studied by Kifer-Ledrappier \cite{Kifer}, Benoist-Hulin \cite{BeHu}. Recently, Cavallucci \cite{Cava} proved that for Hadamard manifolds of pinched non-positive sectional curvature, the (modified) Minkowski dimension of $\partial \X$, the Lipschitz-topological entropy of the geodesic flow, the logarithmic volume entropy and the covering entropy of $\X$, all coincide.

The proof of Theorem \ref{exceptional_sharp_thm} also provides a rigidity result in the similar vein, demonstrating the influence of eigenfunctions on the geometry at infinity:
\begin{corollary} \label{rigidity_cor}
Let $\X$ be a non-positively curved harmonic manifold of purely exponential volume growth. Then the Hausdorff dimension of $\partial \X$ coincides with the mean curvature of horospheres and the logarithmic volume entropy. 
\end{corollary}

In the second part of the paper, we will study the boundary behaviour of positive superharmonic functions. In contrast to the extensive literature on harmonic functions, the boundary behaviour of positive superharmonic functions on non-positively curved Hadamard manifolds, has been studied only in a handful of articles, mostly addressing the radial limits (see \cite{LMT, My_exceptional}). This brings us to the natural question:
\begin{itemize}
\item[($\mathcal{Q}5$)] {\it Like harmonic functions, do positive superharmonic functions also have non-tangential limits a.e. on the boundary at infinity?} 
\end{itemize}

Our next result answers the problem ($\mathcal{Q}5$) in the negative:
\begin{theorem} \label{first_counter-example}
Let $\X$ be a non-positively curved harmonic manifold of purely exponential volume growth. Then given $\alpha>1$, there exists a positive superharmonic function $f$ on $\X$ such that
\begin{equation*}
\limsup_{\substack{x \to \xi \\ x \in \Gamma_\alpha(\xi)}} f(x) = +\infty\:,\:\: \forall \xi \in \partial \X\:.
\end{equation*} 
\end{theorem}

The interesting negative result given by Theorem \ref{first_counter-example} prompts us to seek sufficient conditions to ensure a positive answer to the problem ($\mathcal{Q}5$). The key to this analysis, is the Riesz decomposition \cite[Theorem 1.3]{My_exceptional}: given a positive superharmonic function $f$ on $\X$, there exists a non-negative harmonic function $F_f$ (the greatest harmonic minorant of $f$) and $G[\mu_f]$, the Green potential of $\mu_f$ (the Riesz measure of $f$), such that
\begin{equation*}
f=F_f\:+\:G[\mu_f]\:\:\text{ on } \X\:.
\end{equation*}
Upon a brief glance at the construction in Theorem \ref{first_counter-example}, we realize that the underlying Riesz measure is singular with respect to the volume measure on $\X$. This suggests the following refinement of ($\mathcal{Q}5$):
\begin{itemize}
\item[($\mathcal{Q}6$)] {\it Do positive superharmonic functions with Riesz measures absolutely continuous with respect to the volume measure, have non-tangential limits a.e. on the boundary at infinity?} 
\end{itemize}
Our next result provides an affirmative answer to ($\mathcal{Q}6$) under certain weighted integrability conditions on the density:
\begin{theorem} \label{non-tangential_superharmonic_thm}
Let $\X,\:\partial \X,\:h$ be as in Theorem \ref{Fatou_thm}. Let $f$ be a positive superharmonic function on $\X$ such that the Riesz measure $\mu_f$ is absolutely continuous with respect to the volume measure, with density $\psi$ satisfying
\begin{equation} \label{density_condition}
\int_{\X} e^{-hd(o,x)}\:\psi(x)^p\:dvol(x) < \infty\:,\:\: \text{ for some } p>\frac{n}{2}\:.
\end{equation}
Then the non-tangential limit of $f$ exists at $\lambda_o$-a.e. point on $\partial \X$ and equals the Radon-Nikodym derivative $\frac{d\omega}{d\lambda_o}$, where $\omega$ is the absolutely continuous component of the boundary measure of the greatest harmonic minorant of $f$, with respect to $\lambda_o$.
\end{theorem}

In fact, the integrability condition (\ref{density_condition}) in Theorem \ref{non-tangential_superharmonic_thm} is sharp:
\begin{theorem} \label{second_counter-example}
Let $\X$ be a non-positively curved harmonic manifold of purely exponential volume growth with mean curvature of horospheres $h \ge n/2$. Given $\alpha>1$, there exists a positive superharmonic function $f$ on $\X$ such that its Riesz measure $\mu_f$ is absolutely continuous with respect to the volume measure, with density $\psi$ satisfying
\begin{equation} \label{density_counter_condition}
\int_{\X} e^{-hd(o,x)}\:\psi(x)^p\:dvol(x) < \infty\:,\:\: \text{ for } p \le \frac{n}{2}\:,
\end{equation}
but
\begin{equation} \label{boundary_blowup}
\limsup_{\substack{x \to \xi \\ x \in \Gamma_\alpha(\xi)}} f(x) = +\infty\:,\:\: \forall \xi \in \partial \X\:.
\end{equation}
\end{theorem}

\begin{remark} \label{remark_scaling}
The condition on the mean curvature of horospheres $h \ge n/2$ imposed in Theorem \ref{second_counter-example}, is not restrictive as this can always be achieved by a rescaling of the metric (see \cite[Lemma 5]{Satoh}).
\end{remark}

A natural question now is whether one can talk about boundary limits through wider approach regions. This brings us to the following notion of regions which have tangential contact of degree $\tau>1$ in all directions, at boundary points on $\partial \X$:
\begin{definition} \label{tangential_limit_defn}
\begin{itemize}
\item[(i)] For $\xi \in \partial \X$ and $\alpha >1,\tau>1$, {\it the tangential approach region of degree $\tau$ based at $\xi$ with aperture $\alpha$} is defined to be,  
\begin{equation*}
\Gamma_{\alpha,\tau}(\xi):= \left\{x \in \X \mid e^{-\tau\left(x |\xi\right)_o}< \alpha e^{-d(o,x)}\right\}\:.
\end{equation*}
\item[(ii)] A function $f$ on $\X$ is said to have tangential limit $L$ of degree $\tau>1$ at $\xi \in \partial \X$ if
\begin{equation*}
\displaystyle\lim_{\substack{x \to \xi \\ x \in \Gamma_{\alpha,\tau}(\xi)}} f(x)=L\:,\:\:\forall \alpha >1\:.
\end{equation*}
\end{itemize}
\end{definition}

It is classical in literature and goes back to Littlewood \cite{Littlewood} that harmonic functions may not have tangential limits a.e. on the boundary. In this regard, Nagel-Rudin-Shapiro obtained some positive results on the upper half space $\R^{n+1}_+$ for Poisson integrals of functions belonging to the exotic space of $K$-potentials on $\R^n$ \cite[Theorem 2.9]{Nagel}. This result was subsequently generalized to homogeneous spaces by Cifuentes-Dorronsoro-Sueiro in \cite{Cifuentes}.

Our next result addresses the existence of tangential boundary limits of positive superharmonic functions but does not follow the lines explored in \cite{Nagel,Cifuentes}. Instead, removing the $h$-doubling constraint of Theorem \ref{non-tangential_superharmonic_thm}, we obtain a positive result, with a precise control on the size of the exceptional sets, at the cost of some `extra decay':
\begin{theorem} \label{tangential_limit_superharmonic_thm}
Let $\X$ be a non-positively curved harmonic manifold of purely exponential volume growth with mean curvature of horospheres $h>0$. Let $f$ be a positive superharmonic function on $\X$ such that the Riesz measure $\mu_f$ is absolutely continuous with respect to the volume measure, with density $\psi$ satisfying
\begin{equation} \label{density_condition'}
\int_{\X} e^{-\beta d(o,x)}\:\psi(x)^p\:dvol(x) < \infty\:,\:\: \text{ for some } \beta \in (0,h),\:\:p>\frac{n}{2}\:.
\end{equation}
Also if for some $x_0 \in \X$, the geodesic spherical averages of $f$ asymptotically vanish:
\begin{equation} \label{spherical_average_decay}
\displaystyle\lim_{r \to +\infty} \int_{T^1_{x_0}X} f\left(\gamma_{x_0,v}(r)\right)\: d\theta_{x_0}(v) = 0 \:,
\end{equation}
then for each $\tau \in (1,h/\beta)$, there exists a set $E \subset \partial \X$ with $\mathcal{H}^{\beta \tau}(E)=0$ such that $f$ has tangential limit $0$ of degree $\tau$ at all $\xi \in \partial \X \setminus E\:.$ 
\end{theorem}

We now briefly summarize the key ideas and novelties of this article:
\begin{itemize}
\item The proof of the Fatou theorem for eigenfunctions (Theorem \ref{Fatou_thm}) is standard. It relies on bounding the non-tangential maximal function corresponding to the weighted eigenfunctions, by the Hardy-Littlewood maximal function of the boundary measure on $\partial \X$ (Lemma \ref{maximal_estimate_lemma}), followed by the weak $L^1$ boundedness of the latter (Lemma \ref{Sec3_lemma1}).

\medskip

\item To study the exceptional sets for radial limits of eigenfunctions (Theorems \ref{exceptional_thm} and \ref{exceptional_sharp_thm}), we introduce a truncated maximal function on visual annuli in $\partial \X$ (see (\ref{maximal_fn})). While this approach builds on author's prior work \cite{My_exceptional} for harmonic functions, part (ii) of Theorem \ref{exceptional_sharp_thm}, however, is genuinely new, even for the harmonic case. Its construction crucially uses a non-trivial, abstract result on the existence of sets of prescribed Hausdorff dimensions in complete, separable metric spaces (Lemma \ref{existence}). But since, in general, $\nu$ is only a quasi-metric, we instead transition to visual metrics. We consider visual parameters compatible with the `asymptotic upper curvature bound of $\X$', and then leverage the power scaling and bi-Lipschitz invariance properties of Hausdorff measures. In the process, we also obtain the interesting rigidity result Corollary \ref{rigidity_cor}.

\medskip

\item At the core of studying the boundary behaviour of positive superharmonic functions, is the Riesz decomposition (Lemma \ref{riesz_decomposition_lemma}). Furthermore, by Theorem \ref{Fatou_thm}, examining the non-tangential boundary behaviour of these functions can be simplified by focusing entirely on Green potentials (Theorem \ref{non-tangential_green_thm}). This is then done by decomposing the Green potential into polar and non-polar parts:
\begin{itemize}
\item {\bf the non-polar part:} its non-tangential boundary behaviour (Lemma \ref{non-polar_lemma}) is governed by the Poisson integral of the underlying Radon measure projected onto $\partial \X$, coupled with Lemma \ref{Sec3_lemma1}.
\item {\bf the polar part:} this is controlled by the local integrability of the Green function near its pole and the $L^1(\partial \X)$-integrability of the mass operator on non-tangential cones (Lemma \ref{mass_lemma}).
\end{itemize}
Ultimately, this local integrability yields the $p >n/2$ weighted integrability threshold for the density.

\medskip

\item  While studying the tangential boundary behaviour of positive superharmonic functions in Theorem \ref{tangential_limit_superharmonic_thm}, the asymptotic decay of spherical averages given by (\ref{spherical_average_decay}) ensures that the function is purely a Green potential. Then to estimate the Hausdorff measures of the exceptional sets, we apply Frostman’s lemma (Lemma \ref{frostmann_lemma}) by again transitioning to visual metrics and leveraging the power scaling and bi-Lipschitz invariance of Hausdorff measures. The approach of decomposing into polar and non-polar parts yields two separate descriptions of the exceptional sets:
\begin{itemize}
\item {\bf the non-polar part:} the exceptional sets for tangential limits are identified as the blow-up sets of certain energies  (see Corollary \ref{Cor2}, Lemmata \ref{prop2} and \ref{prop3}). Remarkably, efficient estimates of these energies relate to $\varphi_\lambda$, the spherical functions (see Lemma \ref{energy_estimates}).  
\item {\bf polar part:} the exceptional sets in this case, turn out to be the collection of boundary points where the mass operator on tangential approach regions fails to be finite. By the geometry of these approach regions, this again boils down to the energy estimates mentioned above (see Corollary \ref{Cor3} and part (i) of Lemma \ref{prop3}).
\end{itemize}

\medskip

\item In the counter-examples (Theorems \ref{first_counter-example} and \ref{second_counter-example}), the heart of the matter is to construct a countable set $\{x_j\}_{j=1}^\infty$ in $\X$ which is sufficiently sparse, yet given any $\xi \in \partial \X$, there exists a subsequence $\{x_{j_k}\}_{k=1}^\infty \subset \Gamma_\alpha(\xi)$ converging to $ \xi$. We achieve this using the geometric fact that the sectional curvature of a harmonic manifold is always bounded below by some negative constant. Subsequently, we apply Alexandrov's angle comparison.
\end{itemize}

We conclude this commentary, by noting that Theorems \ref{exceptional_thm} and \ref{exceptional_sharp_thm} on boundary exceptional sets for eigenfunctions, are new even for the simplest case of the real hyperbolic spaces. On the other hand, the results of the second part: Theorems \ref{first_counter-example}, \ref{non-tangential_superharmonic_thm}, \ref{second_counter-example} and \ref{tangential_limit_superharmonic_thm}, were only known for the complex unit ball \cite{Cima, Stoll1, Stoll2} and hence are new even for the homogeneous cases of rank one Riemannian symmetric spaces of non-compact type and Damek-Ricci spaces. While the authors' prior approach for the complex unit ball relied on harmonic analysis of Lie groups, this method breaks down in our general non-homogeneous setting. Instead, we employ potential theory adapted to the intrinsic Gromov hyperbolic geometry of $\X$, a framework consistent with recent literature \cite{Petit, KP16, BKP, Brammen1, Brammen2, My_exceptional}.

\medskip

This article is organized as follows. In Section \ref{sec2}, we fix our notations, recall the required preliminaries on Gromov hyperbolic spaces, harmonic manifolds and Hausdorff measures. In Section \ref{sec3}, we prove Theorem \ref{Fatou_thm}. Theorems \ref{exceptional_thm} and \ref{exceptional_sharp_thm} are proved in Section \ref{sec4}. In Sections \ref{sec5} and \ref{sec6}, we study the non-tangential and tangetial boundary behaviour of Green potentials respectively and prove Theorems \ref{non-tangential_green_thm} and \ref{tangential_limit_green_thm}. In Section \ref{sec7}, we complete the proofs of Theorems \ref{non-tangential_superharmonic_thm} and \ref{tangential_limit_superharmonic_thm}. In Section \ref{sec8}, we construct the counter-examples of Theorems \ref{first_counter-example} and \ref{second_counter-example}. Finally, in Section \ref{sec9}, we conclude by making some remarks and posing some new problems.

\section{Preliminaries}
\label{sec2}
\subsection{Some notations:}
Throughout this article, $C$ will be used to denote positive constants whose value may vary at each occurence, with dependence on parameters or geometric quantities made explicit when necessary. When required, enumerated constants $C_1, C_2, \dots$ will be used to specify fixed constants. 

Let $f_1$ and $f_2$ be two positive functions. Then the notation $f_1 \asymp f_2$ will imply that there exists $C\ge 1$ such that $(1/C) f_1 \le f_2 \le C f_1$. Also $f_1 \gtrsim f_2$ (respectively, $f_1 \lesssim f_2$) will imply that there exists $C \ge 1$ such that $f_1 \ge C f_2$ (respectively, $f_1 \le C f_2$). The indicator function of a set $A$ will be denoted by $\chi_A$. Complex measures are assumed to have their real and imaginary parts to be Radon.

\subsection{Gromov Hyperbolic Spaces:}
In this subsection we recall briefly some  basic facts and definitions related to Gromov hyperbolic spaces. For more details, we refer to \cite{Bridson}.

A {\it geodesic} in a metric space $X$ is an isometric embedding $\gamma : I \subset \mathbb{R} \to X$ of an interval into $X$. A metric space $X$ is said to be a {\it geodesic metric space} if any two points in $X$ can be joined by a geodesic. For $x,y,z \in X$, the Gromov product of $y,z$ with respect to $x$ is defined by,
 \begin{equation*} 
 (y|z)_x := \frac{1}{2} \left(d(x,y)+d(x,z)-d(y,z)\right) \:.
 \end{equation*}
A geodesic metric space $X$ is called {\it Gromov hyperbolic} if there exists a $\delta \ge 0$ such that for every $x,y,z \in X$ and every basepoint $o \in X$, we have 
\begin{equation} \label{four_pt_condition}
(x|y)_o \ge \min\{(x|z)_o\:,\:(y|z)_o\} - \delta\:\:.
\end{equation}
This $\delta$ is called the {\it Gromov hyperbolicity constant}.

For a Gromov hyperbolic space $X$, its {\it boundary at infinity} $\partial X$ is defined to be the set of equivalence classes of geodesic rays in $X$. Here a geodesic ray is an isometric embedding $\gamma : [0,\infty) \to X$ of a closed half-line into $X$, and two geodesic rays $\gamma, \tilde{\gamma}$ are said to be equivalent if the set $\{ d(\gamma(t), \tilde{\gamma}(t)) \ | \ t \geq 0 \}$ is bounded. The equivalence class of a geodesic ray $\gamma$ is denoted by $\gamma(\infty) \in \partial X$.  

A metric space is said to be {\it proper} if closed and bounded balls in the space are compact. Let $X$ be a proper, geodesic, Gromov hyperbolic space. There is a natural topology on $\overline{X} := X \cup \partial X$, called the {\it cone topology} such that $\overline{X}$ is a compact metrizable space which is a compactification of $X$. In this case, for every geodesic ray $\gamma$,  $\gamma(t) \to \gamma(\infty) \in \partial X$ as $t \to \infty$, and for any $x \in X,\: \xi \in \partial X$ there exists a geodesic ray  $\gamma$ such that $\gamma(0) = x, \gamma(\infty) = \xi$.

If the space $X$ is a proper, $CAT(0)$, Gromov hyperbolic space then for any $x \in X$, the Gromov product ${(\cdot|\cdot)}_x$\:, extends continuously to $\partial X \times \partial X$ \cite[Proposition 5.15]{Biswas} and hence we define: 
\begin{equation} \label{Gromov_product_bdry}
 (\xi|\eta)_x := \displaystyle\lim_{\substack{y \to \xi \\ z \to \eta}} (y|z)_x \:.
 \end{equation}
We note that $(\xi|\eta)_x = +\infty$ if and only if $\xi = \eta \in \partial X$. We remark that our definition (\ref{Gromov_product_bdry}) may differ from the usual definition on general Gromov hyperbolic spaces where the limit above may not always exist and instead involves $\liminf$. Moreover, the above boundary continuity of the Gromov product results in the boundary continuity of the Busemann function defined in (\ref{busemann}).

The condition (\ref{four_pt_condition}) extends to $\overline{X}$ as follows,
\begin{equation} \label{four_pt_condition'}
(x|y)_o \ge \min\{(x|z)_o\:,\:(y|z)_o\} - 2\delta\:,\:x,y,z \in \overline{X}\:.
\end{equation}
This leads to the pseudo-triangle inequality on $\partial X$:
\begin{equation*} \label{pseudo_triangle}
\nu(\xi,\zeta) \le e^{2\delta} \left(\nu(\xi,\eta)\:+\:\nu(\eta,\zeta)\right) \:,\:\xi,\zeta,\eta \in \partial X\:,
\end{equation*}
where $\nu$ is defined in (\ref{visual_metric}).

\subsection{Harmonic analysis on harmonic manifolds:}
In this subsection we discuss the structure of harmonic manifolds and harmonic analysis thereon. The materials covered here can be found in \cite{BKP,My_exceptional}. 

Let $\X$ be a non-compact harmonic manifold of purely exponential volume growth, with origin $o \in \X$. For any $v \in  T^1_x \X$ (the unit tangent space at $x$) and $r > 0$, let $A(v,r)$ denote the Jacobian of the map $v \mapsto \exp_x(rv)$ and $\gamma_{x,v}$ denote the geodesic segment starting from $x$ with initial velocity $v$. The definition of a harmonic manifold which has been given in the Introduction is equivalent (\cite[p. 224]{Willmore}) to the fact that this Jacobian is solely a function of the radius, that is, there is a function $A$ on $(0,\infty)$, such that $A(v,r) = A(r)$ for all $v \in T^1_x\X$. This function $A$ is called the {\it volume density} function of $\X$. $A$ satisfies the following asymptotics:
\begin{equation} \label{jacobian_estimate}
A(r) \asymp \begin{cases}
             r^{n-1} & \text{ if } 0<r\le 1 \\
             e^{hr} & \text{ if } r>1 \:.
             \end{cases}
\end{equation}

In \cite{Kn12}, it was shown that for $\X$, a simply connected non-compact harmonic manifold of purely exponential
volume growth, the condition of purely exponential volume growth is equivalent to either of the following conditions:
\begin{enumerate}
\item $\X$ is Gromov hyperbolic.
\item $\X$ has rank one.
\item The geodesic flow of $\X$ is Anosov with respect to the Sasaki metric.
\end{enumerate}

Moreover, the Gromov boundary coincides with the visibility boundary $\partial \X$ introduced in \cite{Eberlein}. This last fact follows from the work in \cite{KP16}. 

One has a family of measures on $\partial \X$ called the visibility measures $\{\lambda_x\}_{x \in \X}$. For $x \in \X$, let $\theta_x$ denote the normalized canonical measure on $T^1_x \X$ , induced by the Riemannian metric and then the visibility measure $\lambda_x$ is obtained as the push-forward of $\theta_x$ to the boundary $\partial \X$ under the radial projection. The visibility measures $\lambda_x$ are pairwise absolutely continuous. For $(x, \xi) \in \X \times \partial \X$, the Poisson kernel is obtained as the following Radon-Nykodym derivative:
\begin{equation*}
P(x, \xi) = e^{-hB_\xi(x)} = \frac{d \lambda_x}{d \lambda_o}(\xi) \:.
\end{equation*}
 
We now talk about the spherical functions on $\X$. For $\lambda \in \C$, $\varphi_\lambda$ is the unique radial eigenfunction of $\Delta$ with eigenvalue $-(\lambda^2+\rho^2)$, normalized to $\varphi_\lambda(o)=1$ and is explicitly given by the integral formula:
\begin{equation} \label{spherical_fn}
\varphi_\lambda(x) = \int_{\partial \X} e^{\left(i\lambda-\rho\right)B_\xi(x)}\:d\lambda_o(\xi)\:\:,
\end{equation}
where we recall that $\rho=h/2$. Now as for each fixed $x \in \X$, $\lambda \mapsto \varphi_\lambda(x)$ is entire in $\C$ and since $\varphi_\lambda(x)=\varphi_{-\lambda}(x)$ for all $\lambda \in \R$, for $\lambda=i\beta$ with $\beta>0$, we have by the Harish-Chandra expansion,
\begin{eqnarray} \label{spherical_fn_estimate}
\varphi_{i\beta}(x)=\varphi_{-i\beta}(x) \asymp {\bf c}(-i\beta)\:e^{\left(\beta-\rho\right)d(o,x)}\:.
\end{eqnarray} 
In the above asymptotics, ${\bf c}(\cdot)$ is the Harish-Chandra's ${\bf c}$ function, which is a holomorphic function in the lower half-plane $\{z \in \C : Im(z)<0\}$ and in particular, is positive on the negative imaginary axis.  

We now provide a brief justification of (\ref{approx_identity}) mentioned in the Introduction. We note by (\ref{spherical_fn}) and (\ref{spherical_fn_estimate}) that
\begin{eqnarray*}
e^{-(\beta - \rho)d(o,x)}\int_{\partial \X} P_{i\beta}(x,\xi) \:d\lambda_o(\xi) &=& C_\beta \: e^{-(\beta - \rho)d(o,x)}\int_{\partial \X} e^{-(\beta +\rho)B_\xi(x)}\:d\lambda_o(\xi) \\
&=& C_\beta \: e^{-(\beta - \rho)d(o,x)} \:\varphi_{i\beta}(x) \\
&\asymp & C_\beta\: {\bf c}(-i\beta)\:.
\end{eqnarray*}
Thus by suitable normalization of $C_\beta$, we indeed get the approximate identity property (\ref{approx_identity}).

We have the following crucial estimate on the visibility measure of closed visual balls \cite[Lemma 6.3]{BKP},
\begin{equation} \label{upper_bound_visual_measure}
\lambda_o\left(\overline{\B(\xi,r)}\right) \lesssim r^h\:,
\end{equation}
where the implicit constant depends only on $h$ and $\delta$.

We recall that for an open subset $\Omega \subset \X$, an upper semi-continuous function $f: \Omega \to [-\infty,+\infty)$, with $f \not \equiv -\infty$ is subharmonic on $\Omega$ if 
\begin{equation} \label{submvp}
f(x) \le \int_{T^1_x \X} f\left(\gamma_{x,v}(r)\right) d\theta_x(v) \:,
\end{equation}
for all $x \in \Omega$ and $r>0$ sufficiently small. It is known that if $f$ is subharmonic on $\X$ then (\ref{submvp}) is true for all $r>0$. Moreover, $f$ is locally integrable and bounded above on compact sets. For $f \in C^2(\X)$, the above notion of subharmonicity is equivalent to the condition that $\Delta f  \ge 0$. A function $f$ is superharmonic if $-f$ is subharmonic.

Now as in our case, the volume density satisfies
\begin{equation*}
\int_1^{+\infty} \frac{1}{A(r)} \: dr < +\infty \:,
\end{equation*}
we have a positive Green function, which is a radial function defined by
\begin{equation*} 
G(r) = \frac{1}{C} \int_{r}^{+\infty} \frac{1}{A(s)} \:ds \:,
\end{equation*}
for some constant $C>0$. Then (\ref{jacobian_estimate}) yields the following estimates of the Green function:
\begin{equation} \label{green_estimate}
G(r) \asymp \begin{cases}
             \frac{1}{r^{n-2}} & \text{ if } 0<r\le 1 \\
             e^{-hr} & \text{ if } r>1 \:,
             \end{cases}
\end{equation}
upto a positive constant depending only on $n$ and $h$. Then for $x \in \X$, the Green function with pole at $x$ is defined by,
\begin{equation*}
G_x(y) := (G \circ d_x)(y)= G(d(x,y)) \:, \text{ for } y \in \X\:,
\end{equation*}
and is denoted by $G(x,y)$. Note that it is symmetric in its arguments. The distributional Laplacian of $G_x$ is, 
\begin{equation*}
\Delta G_x = - \delta_x \:.
\end{equation*}
$G_x$ is harmonic on $\X \setminus \{x\}$ and superharmonic on $\X$. For a Radon measure $\mu$ on $\X$, we say that it has a well-defined Green potential $G[\mu]$, if there exists $x_0 \in \X$ such that
\begin{equation*}
G[\mu](x_0) = \int_\X G(x_0,y)\: d\mu(y) < +\infty \:.
\end{equation*}
If $\mu = f\:dvol$, then we simply write $G[f]$. Green potentials are just special examples of positive superharmonic functions. 

If $f$ is a superharmonic function on $\X$, then there exists a unique Radon measure $\mu_f$ on $\X$ such that 
\begin{equation*}
\int_\X \psi \:d\mu_f = -\int_\X f \Delta \psi \:dvol \:, \text{ for all } \psi \in C^2_c(\X)\:.
\end{equation*}
The above measure $\mu_f$ is called the {\it Riesz measure} of $f$. A {\em harmonic minorant} of a superharmonic function $f$, is a harmonic function $h$ such that $h \le f$ on $\X$. A harmonic function $h$ on $\X$ is said to be the {\it greatest harmonic minorant} of a superharmonic function $f$ on $\X$ if
\begin{itemize}
\item[(i)] $h$ is a harmonic minorant of $f$ and
\item[(ii)] $h(x) \ge H(x)$ for all $x \in \X$, whenever $H$ is a harmonic minorant of $f$\:.
\end{itemize}

We now state the Riesz decomposition theorem of a superharmonic function:
\begin{lemma} \cite[Theorem 1.3]{My_exceptional} \label{riesz_decomposition_lemma}
Let $f$ be a superharmonic function on $\X$ such that it has a harmonic minorant. Then
\begin{equation*}
f(x) = F_f(x) + \int_\X G(x,y) d\mu_f(y) \:,\text{ for all } x \in \X \:,
\end{equation*}
where $F_f$ and $\mu_f$ are the greatest harmonic minorant and the Riesz measure of $f$ respectively.
\end{lemma}

The harmonic function $F_f$ in Lemma \ref{riesz_decomposition_lemma}, has the following explicit form in terms of the asymptotics of geodesic spherical averages of $f$ \cite[pp. 27-28]{My_exceptional}:
\begin{equation} \label{form_harmonic_minorant}
F_f(x) = \displaystyle\lim_{r \to +\infty} \int_{T^1_{x}X} f\left(\gamma_{x,v}(r)\right)\: d\theta_{x}(v) \:,\:\: x \in \X\:.
\end{equation}

Let $\X$ be a Hadamard manifold with sectional curvature $-b^2\le K_\X \le 0$, for some $b>0$. We now talk about comparison angles. If three points in $\X$ lie on the same geodesic, then they are called {\it collinear}. For three points $x,y,z$ which are not collinear, we form the geodesic triangle $\triangle$ in $\X$ by the geodesic segments $[x, y],\: [y, z],\: [z, x]$. A comparison triangle is a geodesic triangle $\overline{\triangle}$ in $\mathbb{H}^2(-b^2)$ formed by geodesic segments $[\overline{x}, \overline{y}],\: [\overline{y}, \overline{z}],\: [\overline{z}, \overline{x}]$ of the same lengths as those of $\triangle$ (such a triangle exists and is unique up to isometry). Let $\theta(y,z)$ denote the Riemannian angle between the points $y$ and $z$, subtended at $x$. The corresponding angle between $\overline{y}$ and $\overline{z}$ subtended at $\overline{x}$ is called the {\it comparison angle} of $\theta(y,z)$ in $\mathbb{H}^2(-b^2)$ and denoted by $\theta_b(y,z)$. Then by Alexandrov's angle comparison theorem,
\begin{equation} \label{finite_angle_comparison}
\theta_b(y,z) \le \theta(y,z) \:.
\end{equation}

\subsection{Hausdorff measure and Hausdorff dimension:}
In the setting of a general quasi-metric space, we now briefly recall the definitions of Hausdorff outer measure, Hausdorff dimensions and some other important results. These can be found in \cite{Falconer,Hei,Mattila}. 

Let $(X,d)$ be a quasi-metric space, i.e., $d$ satisfies positivity, symmetry but instead of the usual triangle inequality, satisfies a pseudo-triangle inequality:
\begin{equation*}
d(x,y) \le C \left(d(x,z)\:+\:d(z,y)\right)\:,\:\:x,y,z \in X\:,
\end{equation*}
for some constant $C \ge 1$. For a set $E \subset X$, $\mathfrak{d}(E)$ will denote the diameter of $E$. For $\varepsilon > 0$, an {\it $\varepsilon$-cover} of a set $E \subset X$ is a countable (or finite) collection of $d$-balls $\{B_i\}$ with 
\begin{equation*}
0 < \mathfrak{d}\left(B_i\right) \le \varepsilon \text{, for all }i \text{ such that } E \subset \displaystyle\bigcup_{i} B_i \:.
\end{equation*}
For $t \ge 0$, we recall that
\begin{equation*}
\mathcal{H}^t_\varepsilon(E) := \inf \left\{\displaystyle\sum_{i} {\left(\mathfrak{d}\left(B_i\right)\right)}^t : \{B_i\} \text{ is an } \varepsilon\text{-cover of } E\right\}\:.
\end{equation*}
Then the {\it $t$-dimensional Hausdorff outer measure} of $E$ is defined by,
\begin{equation*}
\mathcal{H}^t(E) := \displaystyle\lim_{\varepsilon \to 0} \mathcal{H}^t_\varepsilon(E) \:.
\end{equation*}
This outer measure restricts to a measure (also denoted by) $\mathcal{H}^t$ on a $\sigma$-algebra that contains all Borel sets. It is called the {\it $t$-dimensional Hausdorff measure}.

The {\it Hausdorff dimension} of $E$ is defined by
\begin{equation*}
dim_{\mathcal{H}}E := \inf \left\{t \ge 0 : \mathcal{H}^t(E) < +\infty\right\} \:.
\end{equation*}

We will require the following result on existence of sets of desired Hausdorff dimensions, which is a special case of \cite[Corollary 7]{Howroyd}:
\begin{lemma} \label{existence}
If $Y$ is a Borel subset of a complete, separable metric space such that for some $t>0$, $\mathcal{H}^t(Y)=\infty$, then there exists a compact subset $Z \subset Y$ with $dim_{\mathcal{H}}Z=t$.
\end{lemma}

Another crucial tool in geometric measure theory is the Frostmann's lemma \cite[Theorem 8.17]{Mattila}:
\begin{lemma}  \label{frostmann_lemma}
Let $K$ be a compact subset of a compact metric space $(X,d)$. If $\mathcal{H}^t(K)>0$, then $K$ contains the support of a positive measure $\mu$ satisfying,
\begin{equation*}
\mu\left(B(x,r)\right) \le Cr^t\:,\: \forall x \in X,\:r>0\:,
\end{equation*}
for some constant $C>0$\:.
\end{lemma}

\section*{Part I: Boundary behaviour of eigenfunctions}

\section{Non-tangential limits}
\label{sec3}
In this section, our goal is to prove Theorem \ref{Fatou_thm} and hence we work under the hypothesis of Theorem \ref{Fatou_thm}. In this direction, let us recall the Hardy-Littlewood maximal function of complex measures $\mu$ on $\partial \X$, in terms of their total variation $|\mu|$:
\begin{equation*}
\M_{HL}\mu(\xi):=\displaystyle\sup_{0<r<1} \frac{|\mu|\left(\B(\xi,r)\right)}{\lambda_o\left(\B(\xi,r)\right)}\:,\:\:\xi \in \partial \X\:.
\end{equation*}
If $d\mu =f\:d\lambda_o$, then we simply write $\M_{HL}f$.

We first obtain the following estimate on the visibility measures of superlevel sets of the maximal function:
\begin{lemma} \label{Sec3_lemma1}
For any complex measure $\mu$ on $\partial \X$ and any $t>0$, we have
\begin{equation*}
\lambda_o\left(\left\{\xi \in \partial \X : \M_{HL}\mu(\xi)>t\right\}\right) \lesssim \frac{|\mu|\left(\partial \X\right)}{t}\:,
\end{equation*}
where the implicit constant depends only on the intrinsic geometry of $\X$\:.
\end{lemma}
\begin{proof}
Let $E:=\left\{\xi \in \partial \X : \M_{HL}\mu(\xi)>t\right\}$. Then for each $\xi \in E$, there exists $r_\xi \in (0,1)$ such that 
\begin{equation} \label{Sec3_lemma1_eq1}
|\mu|\left(\B\left(\xi,r_\xi\right)\right) > t\: \lambda_o\left(\B\left(\xi,r_\xi\right)\right)\:.
\end{equation}
Then by Vitali $5$-covering Lemma for quasi-metric spaces, there exist countably many visual balls $\{\mathscr{B}\left(\xi_j, r_{\xi_j}\right)\}_{j=1}^\infty$ satisfying (\ref{Sec3_lemma1_eq1}) such that
\begin{itemize}
\item $\mathscr{B}\left(\xi_j, r_{\xi_j}\right) \cap \mathscr{B}\left(\xi_k, r_{\xi_k}\right) = \emptyset$ for all $j \ne k$,
\item $E \subset \displaystyle\bigcup_{j=1}^\infty \mathscr{B}\left(\xi_j, 5e^{4\delta}r_{\xi_j}\right)$ \:.
\end{itemize} 
Thus by the $h$-doubling hypothesis,
\begin{equation*}
\lambda_o(E) \le \sum_{j=1}^\infty \lambda_o\left(\mathscr{B}\left(\xi_j, 5e^{4\delta}r_{\xi_j}\right)\right) \lesssim \sum_{j=1}^\infty \lambda_o\left(\mathscr{B}\left(\xi_j, r_{\xi_j}\right)\right)\:.
\end{equation*} 
Now an application of (\ref{Sec3_lemma1_eq1}) completes the proof. Indeed,
\begin{eqnarray*}
\sum_{j=1}^\infty \lambda_o\left(\mathscr{B}\left(\xi_j, r_{\xi_j}\right)\right) < \frac{1}{t}\sum_{j=1}^\infty |\mu|\left(\mathscr{B}\left(\xi_j, r_{\xi_j}\right)\right) = \frac{1}{t} |\mu|\left(\bigcup_{j=1}^\infty \mathscr{B}\left(\xi_j, r_{\xi_j}\right)\right) \le \frac{|\mu|\left(\partial \X\right)}{t}\:.
\end{eqnarray*}
\end{proof}

As a special case of Lemma \ref{Sec3_lemma1}, we obtain the weak $L^1$ boundedness of the maximal function:
\begin{corollary} \label{Sec3_corollary2}
For $f \in L^1(\partial \X)$ and $t>0$, we have
\begin{equation*}
\lambda_o\left(\left\{\xi \in \partial \X : \M_{HL}f(\xi)>t\right\}\right) \lesssim \frac{\|f\|_{L^1(\partial \X)}}{t}\:,
\end{equation*}
where the implicit constant depends only on the intrinsic geometry of $\X$\:.
\end{corollary}
By a standard application of the method of maximal functions, one obtains the Lebesgue differentiation theorem from Corollary \ref{Sec3_corollary2}:
\begin{corollary} \label{Sec3_corollary3}
If $f \in L^1(\partial \X)$, then for $\lambda_o$-a.e. $\xi \in \partial \X$, we have
\begin{equation*}
\lim_{r \to 0+}\frac{1}{\lambda_o\left(\B(\xi,r)\right)} \int_{\B(\xi,r)} |f(\eta)-f(\xi)|\:d\lambda_o(\eta)=0\:.
\end{equation*}
Hence,
\begin{equation*}
f(\xi)= \lim_{r \to 0+}\frac{1}{\lambda_o\left(\B(\xi,r)\right)} \int_{\B(\xi,r)} f(\eta)\:d\lambda_o(\eta)\:,\:\: \lambda_o\text{-a.e. } \xi \in \partial \X\:.
\end{equation*}
\end{corollary}
\begin{remark} \label{Sec3_remark4}
The points $\xi \in \partial \X$ for which the conclusions in Corollary \ref{Sec3_corollary3} are valid, are called {\it Lebesgue points} of $f$. 
\end{remark}

We now introduce the important notion of the maximal function in  non-tangential cones, which were defined in Definition \ref{non-tangential_limit_defn}:
\begin{definition} \label{non-tangential_max_fn}
Let $u$ be a continuous function on $\X$. For $\alpha>1$, the {\it non-tangential maximal function} of $u$, is defined on $\partial \X$ by,
\begin{equation*}
\M_{\alpha}u(\xi):=\sup\left\{|u(x)|:x \in \Gamma_\alpha(\xi)\right\}\:,\:\:\xi \in \partial \X\:.
\end{equation*}
\end{definition} 

The following estimate relates the non-tangential maximal function of weighted generalized Poisson integrals and the Hardy-Littlewood maximal function of the boundary measure and is the heart of the matter in the proof of Theorem \ref{Fatou_thm}:
\begin{lemma} \label{maximal_estimate_lemma}
For $\beta>0, \alpha >1$ and any complex measure $\mu$ on $\partial \X$,
\begin{equation*}
\M_\alpha\left(e^{-(\beta - \rho)d(o,\cdot)}\PP_{i\beta}[\mu]\right) \lesssim \M_{HL}\mu\:,
\end{equation*}
where the implicit constant depends only on $\alpha, \beta$ and the intrinsic geometry of $\X$\:. 
\end{lemma}
\begin{proof}
Let $\xi \in \partial \X$ and without loss of generality, let us assume that $\M_{HL}\mu(\xi)<\infty$. We fix $x \in \Gamma_\alpha(\xi)$ and set $r:=\alpha e^{-d(o,x)}$. Then there exists a smallest integer $m$ such that $2^mr>1$. Let
\begin{eqnarray*}
E_0 &:=& \B(\xi,r)\:,\\
E_j &:=& \B(\xi,2^jr) \setminus \B(\xi,2^{j-1}r)\:,\:1 \le j \le m\:. 
\end{eqnarray*}
We start by noting that
\begin{equation*}
\left|e^{-(\beta - \rho)d(o,x)}\PP_{i\beta}[\mu](x)\right| \le C_\beta \sum_{j=0}^m I_{\beta,j}(x)\:,
\end{equation*}
where
\begin{equation*}
I_{\beta,j}(x) = e^{-(\beta - \rho)d(o,x)} \int_{E_j} e^{-(\beta + \rho)B_\eta(x)}\:d|\mu|(\eta)\:.
\end{equation*}
First for $\eta \in E_0=\B(\xi,r)$, as by triangle inequality, $- B_\eta(x) \le d(o,x)$, we get that
\begin{eqnarray*}
I_{\beta,0}(x) &=& e^{-(\beta - \rho)d(o,x)} \int_{\B(\xi,r)} e^{-(\beta + \rho)B_\eta(x)}\:d|\mu|(\eta) \\
&\le & e^{hd(o,x)}\:|\mu|\left(\B(\xi,r)\right) \\
&=& \frac{\alpha^h}{r^h}\:|\mu|\left(\B(\xi,r)\right)\:.
\end{eqnarray*}
Then by (\ref{upper_bound_visual_measure}), we get
\begin{equation} \label{maximal_estimate_pf_eq1}
I_{\beta,0}(x) \lesssim \frac{|\mu|\left(\B(\xi,r)\right)}{\lambda_o\left(\B(\xi,r)\right)} \le \M_{HL}\mu(\xi)\:,
\end{equation}
where the implicit constant only depends on $\alpha$ and the intrinsic geometry of $\X$\:.

We next note by (\ref{four_pt_condition'}) that for any $\eta \in \partial \X$,
\begin{equation*}
e^{-(\xi|\eta)_o} \le e^{2\delta} \left(e^{-(x|\xi)_o}\:+\:e^{-(x|\eta)_o}\right)\:.
\end{equation*}
Now as $x \in \Gamma_\alpha(\xi)$, we get that
\begin{equation*}
e^{-(\xi|\eta)_o} < e^{2\delta} \left(\alpha e^{-d(o,x)}\:+\:e^{-(x|\eta)_o}\right)\:.
\end{equation*}
Then as by triangle inequality,
\begin{equation*}
(x|\eta)_o \le d(o,x)\:,
\end{equation*}
it follows that
\begin{equation*}
e^{-(\xi|\eta)_o} < e^{2\delta} (\alpha +1)e^{-(x|\eta)_o}\:.
\end{equation*}
Thus for $\eta \in E_j,\:j \ge 1$, we have that
\begin{equation*}
2^{j-1}r \le \nu\left(\xi,\eta\right)=e^{-(\xi|\eta)_o} < e^{2\delta} (\alpha +1)e^{-(x|\eta)_o}\:,
\end{equation*}
that is,
\begin{equation} \label{maximal_estimate_pf_eq2}
e^{(x|\eta)_o} < \frac{e^{2\delta} (\alpha +1)}{2^{j-1}r}\:.
\end{equation}
We now estimate $I_{\beta,j}(x),\:j\ge 1$, by expanding the Busemann function into distance and the Gromov product and then plugging in (\ref{maximal_estimate_pf_eq2}) to get,
\begin{eqnarray*}
I_{\beta,j}(x) = e^{-2\beta d(o,x)} \int_{E_j} e^{2(\beta+\rho)(x|\eta)_o}\:d|\mu|(\eta) \lesssim  \frac{r^{2\beta}}{\alpha^{2\beta}} \frac{(\alpha +1)^{2(\beta +\rho)}}{\left(2^{j-1}\right)^{2(\beta+\rho)}r^{2(\beta+\rho)}} |\mu|\left(\B(\xi,2^jr)\right)\:, 
\end{eqnarray*}
where the implicit constant only depends on $\beta$ and the intrinsic geometry of $\X$. Now as $\alpha>1$, the right hand side can be further bounded by,
\begin{equation*}
I_{\beta,j}(x) \lesssim \frac{\alpha^h}{\left(2^{2\beta}\right)^{j}(2^jr)^{h}} |\mu|\left(\B(\xi,2^jr)\right)\:.
\end{equation*} 
with the implicit constant depends only on $\beta$ and the intrinsic geometry of $\X$. Finally by (\ref{upper_bound_visual_measure}), we obtain
\begin{equation} \label{maximal_estimate_pf_eq3}
I_{\beta,j}(x) \lesssim \frac{1}{\left(2^{2\beta}\right)^{j}} \frac{|\mu|\left(\B(\xi,2^jr)\right)}{\lambda_o\left(\B(\xi,2^jr)\right)} \le \frac{1}{\left(2^{2\beta}\right)^{j}}\:\M_{HL}\mu(\xi)\:,
\end{equation}
where the implicit constant depends only on $\alpha$ and the intrinsic geometry of $\X$.

Now summing up the estimates (\ref{maximal_estimate_pf_eq1}) and (\ref{maximal_estimate_pf_eq3}), the result follows. 
\end{proof}

We next introduce the notion of strong derivative of measures. 
\begin{definition} \label{Sec3_defn7}
The strong derivative of a Radon measure $\mu$ on $\partial \X$ is defined as,
\begin{equation*}
D\mu(\xi):= \lim_{r \to 0+} \frac{\mu\left(\B(\xi,r)\right)}{\lambda_o\left(\B(\xi,r)\right)}\:,\:\:\xi \in \partial \X\:,
\end{equation*}
provided the limit exists.
\end{definition}

Our next lemma shows that strong derivative of singular measures vanish almost everywhere:
\begin{lemma} \label{Sec3_lemma8}
If $\mu$ is a Radon measure on $\partial \X$ which is singular with respect to $\lambda_o$, then 
\begin{equation*}
D\mu=0\:,\:\:\lambda_o\text{-a.e. on } \partial \X\:. 
\end{equation*}
\end{lemma}
\begin{proof}
We argue by contradiction. If the conclusion of Lemma \ref{Sec3_lemma8} is not true then there exists a positive constant $a$ such that the Borel set
\begin{equation*}
E:=\left\{\xi \in \partial \X : \limsup_{r \to 0+} \frac{\mu\left(\B(\xi,r)\right)}{\lambda_o\left(\B(\xi,r)\right)} > a \right\}\:,
\end{equation*}
satisfies $\lambda_o(E)>0$. Now as $\mu$ is singular with respect to $\lambda_o$, we have that $\mu(E)=0$. Let $\varepsilon >0$. Then there exists an open set $U$ such that $E \subset U$ with $\mu(U)<\varepsilon$. Now for every $\xi \in E$, there exists $r_\xi \in (0,1)$ such that
\begin{equation} \label{Sec3_Lemma8_eq1}
\mu\left(\B(\xi,r_\xi)\right) >a\:\lambda_o\left(\B(\xi,r_\xi)\right)\:,\text{ with } \B(\xi,r_\xi) \subset U\:.
\end{equation}
Again by Vitali $5$-covering Lemma for quasi-metric spaces, there exist countably many visual balls $\{\mathscr{B}\left(\xi_j, r_{\xi_j}\right)\}_{j=1}^\infty$ satisfying (\ref{Sec3_Lemma8_eq1}) such that
\begin{itemize}
\item $\mathscr{B}\left(\xi_j, r_{\xi_j}\right) \cap \mathscr{B}\left(\xi_k, r_{\xi_k}\right) = \emptyset$ for all $j \ne k$,
\item $E \subset \displaystyle\bigcup_{j=1}^\infty \mathscr{B}\left(\xi_j, 5e^{4\delta}r_{\xi_j}\right)$ \:.
\end{itemize}
Utilizing the $h$-doubling hypothesis, we now have
\begin{equation*}
\lambda_o(E) \le \sum_{j=1}^\infty \lambda_o\left(\mathscr{B}\left(\xi_j, 5e^{4\delta}r_{\xi_j}\right)\right) \lesssim \sum_{j=1}^\infty \lambda_o\left(\mathscr{B}\left(\xi_j, r_{\xi_j}\right)\right)\:. 
\end{equation*} 
Then by (\ref{Sec3_Lemma8_eq1}),
\begin{eqnarray*}
\sum_{j=1}^\infty \lambda_o\left(\mathscr{B}\left(\xi_j, r_{\xi_j}\right)\right) < \frac{1}{a} \sum_{j=1}^\infty \mu\left(\mathscr{B}\left(\xi_j, r_{\xi_j}\right)\right) &=& \frac{1}{a}\:\mu\left(\displaystyle\bigcup_{j=1}^\infty \B(\xi_j,r_{\xi_j})\right)\\
&\le &  \frac{1}{a}\: \mu\left(U\right) \\
&<& \frac{\varepsilon}{a}\:.
\end{eqnarray*}
Now as $\varepsilon >0$ was arbitrary, we get that $\lambda_o(E)=0$, which is a contradiction.
\end{proof}

We now consider boundary limits of generalized Poisson integrals where the strong derivative of the boundary measure vanishes:
\begin{lemma} \label{Sec3_lemma9}
Let $\beta>0$. If $\mu$ is a Radon measure on $\partial \X$ such that for some $\xi \in \partial \X$, $D\mu(\xi)=0,$ then the non-tangential limit of  $e^{-(\beta - \rho)d(o,\cdot)}\PP_{i\beta}[\mu]$ exists and equals $0$ at $\xi$.
\end{lemma}
\begin{proof}
Let $\varepsilon>0$ be arbitrary. Since, $D\mu(\xi)=0$, there exists $r_0>0$ such that for all $r \in (0,r_0)$,
\begin{equation} \label{Sec3_lemma9_eq1}
\mu\left(\B(\xi,r)\right) < \varepsilon\: \lambda_o \left(\B(\xi,r)\right)\:.
\end{equation}
Letting $\B_0=\B(\xi,r_0)$, we set
\begin{equation*}
\mu_0:= \mu|_{\B_0}\:\:\:\text{   and     } \mu_1:=\mu-\mu_0\:.
\end{equation*} 
If $\eta \in \partial \X \setminus \B_0$, then 
\begin{equation} \label{Sec3_lemma9_eq2}
e^{-(\xi|\eta)_o} = \nu(\xi,\eta) \ge r_0\:.
\end{equation}
Also as in the proof of Lemma \ref{maximal_estimate_lemma}, for $x \in \Gamma_\alpha(\xi),\:\alpha >1$, we have
\begin{equation} \label{Sec3_lemma9_eq3}
e^{-(\xi|\eta)_o} < e^{2\delta} (\alpha +1)e^{-(x|\eta)_o}\:.
\end{equation}
Thus for $x \in \Gamma_\alpha(\xi)$ and $\eta \in \partial \X \setminus \B_0$, expanding the Busemann function into distance and the Gromov product and then plugging in (\ref{Sec3_lemma9_eq2}) and (\ref{Sec3_lemma9_eq3}), we obtain
\begin{eqnarray*}
e^{-(\beta - \rho)d(o,x)}P_{i\beta}(x,\eta)= C_\beta\:e^{-2\beta d(o,x)}  e^{2(\beta+\rho)(x|\eta)_o} < C_\beta\:\left(\frac{e^{2\delta} (\alpha +1)}{r_0}\right)^{2(\beta+\rho)}\:e^{-2\beta d(o,x)}\:,
\end{eqnarray*}
and thus the non-tangential limit of $e^{-(\beta - \rho)d(o,\cdot)}\PP_{i\beta}[\mu_1]$ vanishes at $\xi$. On the other hand, by (\ref{Sec3_lemma9_eq1}), we note that $\M_{HL}\mu_0 < \varepsilon$. Thus by Lemma \ref{maximal_estimate_lemma}, for any $\alpha>1$, 
\begin{equation*}
\limsup_{\substack{x \to \xi \\ x \in \Gamma_\alpha(\xi)}} e^{-(\beta - \rho)d(o,x)}\PP_{i\beta}[\mu_0](x) \lesssim \varepsilon\:,
\end{equation*}
where the implicit constant depends only on $\alpha, \beta$ and the intrinsic geometry of $\X$\:. The result now follows.
\end{proof}
This brings us to,
\begin{lemma} \label{Sec3_lemma10}
Let $\beta >0$.
\begin{itemize}
\item[(i)] If $f \in L^1(\partial \X)$, then the non-tangential limit of $e^{-(\beta - \rho)d(o,\cdot)}\PP_{i\beta}[f]$ exists and equals $f$ at $\lambda_o$-a.e. on $\partial \X$.
\item[(ii)] If $\mu$ is a complex measure on $\partial \X$ which is singular with respect to $\lambda_o$, then the non-tangential limit of $e^{-(\beta - \rho)d(o,\cdot)}\PP_{i\beta}[\mu]$ exists and equals $0$ at $\lambda_o$-a.e. on $\partial \X$.
\end{itemize}
\end{lemma}
\begin{proof}
We first prove part (ii). Since $\mu$ is singular with respect to $\lambda_o$, so is its total variation $|\mu|$. Then by Lemma \ref{Sec3_lemma8}, we get that $D|\mu|=0,\:\lambda_o$-a.e. on $\partial \X$. Thus by Lemma \ref{Sec3_lemma9}, the non-tangential limit of $e^{-(\beta - \rho)d(o,\cdot)}\PP_{i\beta}[|\mu|]$ exists and equals $0$ at $\lambda_o$-a.e. on $\partial \X$. The result now follows as $$\left|e^{-(\beta - \rho)d(o,\cdot)}\PP_{i\beta}[\mu]\right| \le e^{-(\beta - \rho)d(o,\cdot)}\PP_{i\beta}[|\mu|]\:.$$

\medskip

We now focus on (i). For $f \in L^1(\partial \X)$, by the Lebesgue differentiation theorem (Corollary \ref{Sec3_corollary3}), we have
\begin{equation*}
\lim_{r \to 0+}\frac{1}{\lambda_o\left(\B(\xi,r)\right)} \int_{\B(\xi,r)} |f(\eta)-f(\xi)|\:d\lambda_o(\eta)=0\:,\:\:\lambda_o\text{-a.e. } \xi \in \partial \X\:.
\end{equation*}
We fix  any such $\xi$ and define a measure $\mu_\xi$ on $\partial \X$ by,
\begin{equation*}
\mu_\xi(E):= \int_{E} |f(\eta)-f(\xi)|\:d\lambda_o(\eta)\:,
\end{equation*}
for all Borel $E \subset \partial \X$. Then $D\mu_\xi(\xi)=0$ and 
\begin{equation*}
\left|e^{-(\beta - \rho)d(o,x)}\PP_{i\beta}[f](x)-f(\xi)\right| \le e^{-(\beta - \rho)d(o,x)}\PP_{i\beta}[\mu_\xi](x)\:.
\end{equation*}
The result now follows from Lemma \ref{Sec3_lemma9}. This completes the proof of Lemma \ref{Sec3_lemma10}.
\end{proof}

We now conclude this section by completing the proof of Theorem \ref{Fatou_thm}.
\begin{proof}[Proof of Theorem \ref{Fatou_thm}]
Theorem \ref{Fatou_thm} follows at once from Lemma \ref{Sec3_lemma10} by considering the Lebesgue-Radon-Nikodym decomposition (see \cite[Theorem 6.10]{Rudin}) of $\mu$ with respect to $\lambda_o$.
\end{proof}

\section{Exceptional sets for radial limits}
\label{sec4}
In this section, we prove Theorems \ref{exceptional_thm} and \ref{exceptional_sharp_thm} in Subsections \ref{subsec4.2} and \ref{subsec4.3} respectively. The key to this analysis is the estimates of a truncated maximal function, which we prove now in Subsection \ref{subsec4.1}. 

\subsection{Estimates of a truncated Maximal Function}
\label{subsec4.1}
Let $0< r_1 < r_2 \le 1$ and $\xi \in \partial \X$. Then for a complex measure $\mu$ on $\partial \X$, we consider the following truncated maximal function:
\begin{equation} \label{maximal_fn}
\M_{r_1,r_2}[\mu](\xi) := \displaystyle\sup_{r_1 \le r \le r_2} \frac{|\mu|(\mathscr{B}(\xi,r))}{r^h} \:.
\end{equation}
When $d\mu = f d\lambda_o$ for some suitable function $f$ on $\partial \X$, we will denote the corresponding maximal function simply by $\M_{r_1,r_2}[f]$. 

\medskip

Next we see an estimate relating the weighted generalized Poisson integral of a complex measure, along radial geodesic rays, with the maximal function corresponding to the boundary measure. 
\begin{lemma} \label{maximal_fn_lem}
Let $\beta>0,\:\tau \ge 1,\: 0<\varepsilon \le 1$ and $\mu$ be a complex measure on $\partial \X$. Then for all $t > \log(\tau / \varepsilon)$, one has for all $\xi \in \partial \X$,
\begin{equation} \label{maximal_fn_ineq}
\left|e^{-(\beta - \rho)t}\PP_{i\beta}[\mu](\gamma_\xi(t))\right| \lesssim \left\{e^{ht}|\mu|\left(\mathscr{B}\left(\xi,\tau e^{-t}\right)\right) + \frac{\M_{\tau e^{-t},\varepsilon}[\mu](\xi)}{\tau^{2\beta}}  + \frac{e^{-2\beta t}}{\varepsilon^{h+2\beta}}|\mu|\left(\partial \X\right)\right\} \:,
\end{equation}
where the implicit constant depends only on $\beta$ and the intrinsic geometry of $\X$.
\end{lemma}
\begin{proof}
Fix $\xi \in \partial \X$ and $t > \log(\tau / \varepsilon)$. Then note that $\tau e^{-t} < \varepsilon$. Hence there exists a largest non-negative integer $m$ such that
\begin{equation*}
2^m \tau e^{-t} \le \varepsilon \:.
\end{equation*}
Let 
\begin{eqnarray*}
&E_0& = \mathscr{B}\left(\xi,\tau e^{-t}\right) \:, \\
&E_j& = \mathscr{B}\left(\xi, 2^j \tau e^{-t}\right) \setminus \mathscr{B}\left(\xi, 2^{j-1} \tau e^{-t}\right) ,\:\text{for } 1 \le j \le m \:,\\
&E_{m+1}& = \partial \X \setminus \mathscr{B}\left(\xi, 2^m \tau e^{-t}\right) \:.
\end{eqnarray*}
Now 
\begin{equation*}
\left|e^{-(\beta - \rho)t}\PP_{i\beta}[\mu](\gamma_\xi(t))\right| \le C_\beta \displaystyle\sum_{j=0}^{m+1} I_{\beta,j}(t) \:,
\end{equation*}
where 
\begin{equation*}
I_{\beta,j}(t) = e^{-(\beta-\rho)t}\int_{E_j} e^{-(\beta+\rho)B_{\eta}\left(\gamma_{\xi}(t)\right)} \: d|\mu|(\eta) \:,\:\text{for } 0 \le j \le m+1\:.
\end{equation*}
As by the triangle inequality, $-B_{\eta}\left(\gamma_{\xi}(t)\right) \le t$, we have 
\begin{equation*}
I_{\beta,0}(t) \le  e^{ht} \:|\mu|\left(\mathscr{B}\left(\xi,\tau e^{-t}\right)\right) \:. 
\end{equation*}

Next for $1 \le j \le m$, by proceeding as in the proof of Lemma \ref{maximal_estimate_lemma}, we get that for $\eta \in E_j$,
\begin{equation*}
2^{j-1} \tau e^{-t} \le \nu (\xi,\eta) = e^{-(\xi|\eta)_o} \le e^{2\delta} \left(e^{-(\gamma_\xi(t)|\xi)_o}\:+\:e^{-(\gamma_\xi(t)|\eta)_o}\right)\:.
\end{equation*}
Now as $(\gamma_\xi(t)|\xi)_o=t$ and by triangle inequality, $(\gamma_\xi(t)|\eta)_o \le t$, we get that
\begin{equation*}
2^{j-1} \tau e^{-t} \le e^{2\delta} \left(e^{-t}\:+\:e^{-(\gamma_\xi(t)|\eta)_o}\right) \le 2\:e^{2\delta}\:e^{-(\gamma_\xi(t)|\eta)_o}\:,
\end{equation*}
that is,
\begin{equation} \label{maximal_fn_ineq_eq1}
e^{(\gamma_\xi(t)|\eta)_o} \le \frac{2\:e^{2\delta}}{2^{j-1} \tau e^{-t}}\:.
\end{equation}
Then again proceeding as in the proof of Lemma \ref{maximal_estimate_lemma}, by expanding the Busemann function into distance and the Gromov product and then plugging in (\ref{maximal_fn_ineq_eq1}), we obtain
\begin{eqnarray*}
I_{\beta,j}(t) &\lesssim & \frac{e^{-2\beta t}}{\left(2^{j-1} \tau e^{-t}\right)^{2(\beta +\rho)}}\: |\mu|\left(\B(\xi,2^j \tau e^{-t})\right) \\
&\lesssim & \frac{|\mu|\left(\B(\xi,2^j \tau e^{-t})\right)}{(2^j)^{2\beta} \tau^{2\beta}(2^j\tau e^{-t})^h} \\
& \le & \frac{1}{(2^j)^{2\beta}} \frac{\M_{\tau e^{-t},\varepsilon}[\mu](\xi)}{\tau^{2\beta}}\:,
\end{eqnarray*}
where the implicit constants depend only on $\beta$ and the intrinsic geometry of $\X$. Therefore, 
\begin{equation*}
\displaystyle\sum_{j=1}^m I_{\beta,j}(t) \le \left(\displaystyle\sum_{j=1}^m \frac{1}{{\left(2^{j}\right)}^{2\beta}}\right) \frac{\M_{\tau e^{-t},\varepsilon}[\mu](\xi)}{\tau^{2\beta}} \lesssim  \frac{\M_{\tau e^{-t},\varepsilon}[\mu](\xi)}{\tau^{2\beta}}\:.
\end{equation*}
Repeating the same argument as above, we get
\begin{equation} \label{maximal_fn_ineq_last_eq2}
I_{\beta, m+1}(t) \lesssim \int_{E_{m+1}} \frac{e^{-2\beta t}}{{\left(2^{m} \tau e^{-t}\right)}^{2(\beta+\rho)}} \:d|\mu|(\eta) \:.
\end{equation}
Now by the choice of $m$, 
\begin{equation*}
2^m \tau e^{-t} > \frac{\varepsilon}{2} \:.
\end{equation*}
Plugging the above in (\ref{maximal_fn_ineq_last_eq2}), it follows that
\begin{equation*}
I_{\beta,m+1}(t) \lesssim  \frac{ e^{-2\beta t}}{\varepsilon^{h+2\beta}}\:|\mu|\left(\partial \X\right) \:,
\end{equation*}
where the implicit constant depends only on $\beta$ and the intrinsic geometry of $\X$. Then summing up the above estimates, we get (\ref{maximal_fn_ineq}).
\end{proof}
Lemma \ref{maximal_fn_lem} has the following consequences.
\begin{corollary} \label{cor1}
Let $\beta>0,\:0<\varepsilon \le 1$ and $\mu$ be a complex measure on $\partial \X$. Then for all $t > \log(1 / \varepsilon)$, one has for all $\xi \in \partial \X$,
\begin{equation*} 
\left|e^{-(\beta - \rho)t}\PP_{i\beta}[\mu](\gamma_\xi(t))\right| \lesssim  \left\{ 2\M_{ e^{-t},\varepsilon}[\mu](\xi)  + \frac{e^{-2\beta t}}{\varepsilon^{h+2\beta}} |\mu|(\partial \X) \right\} \:,
\end{equation*}
where the implicit constant depends only on $\beta$ and the intrinsic geometry of $\X$.
\end{corollary}
\begin{proof}
The Corollary follows by taking $\tau=1$ in Lemma \ref{maximal_fn_lem}.
\end{proof}
\begin{corollary} \label{cor2}
Let $\xi \in \partial \X,\:\tau > 1$ and $t>\log(\tau)$. If $f$ is a non-negative measurable function on $\partial \X$ such that $f \equiv 1$ on $\mathscr{B}\left(\xi,\tau e^{-t}\right)$ and $f \le 1$ on $\partial \X$, then there exists $C_1>0$, depending only on $\beta$ and the intrinsic geometry of $\X$ such that
\begin{equation*}
e^{-(\beta - \rho)t}\PP_{i\beta}[f](\gamma_\xi(t)) \ge 1 - \frac{C_1}{\tau^{2\beta}} \:.
\end{equation*}
\end{corollary}
\begin{proof}
Let $t>\log(\tau)$. We consider 
\begin{equation*}
g:= 1-f \:.
\end{equation*}
Then $g$ is a measurable, non-negative function on $\partial \X$ such that
\begin{equation*}
g \equiv 0 \text{ on } \mathscr{B}\left(\xi,\tau e^{-t}\right) \text{ and } g \le 1 \text{ on } \partial \X\:.
\end{equation*}
Then applying Lemma \ref{maximal_fn_lem}, for $d\mu = g\: d\lambda_o$ and $\varepsilon=1$, we get that there exists $C>0$, depending only on $\beta$ and the intrinsic geometry of $\X$ such that 
\begin{eqnarray} \label{cor2_eq}
e^{-(\beta - \rho)t}\PP_{i\beta}[g](\gamma_\xi(t)) &\le & C {\left(\frac{1}{\tau}\right)}^{2\beta} \M_{\tau e^{-t},1}[\mu](\xi) \nonumber\\
&\le & C {\left(\frac{1}{\tau}\right)}^{2\beta} \M_{\tau e^{-t},1}[\lambda_o](\xi)\:.
\end{eqnarray}
Now working with the definition (\ref{maximal_fn}), it follows by applying the estimate (\ref{upper_bound_visual_measure}),
\begin{equation*}
\M_{\tau e^{-t},1}[\lambda_o](\xi) = \displaystyle\sup_{\tau e^{-t} \le r \le 1} \frac{\lambda_o\left(\mathscr{B}(\xi,r)\right)}{r^h} \lesssim 1 \:.
\end{equation*}
Then plugging the above in (\ref{cor2_eq}), one has for some $C>0$, depending only on $\beta$ and the intrinsic geometry of $\X$ such that
\begin{equation*}
e^{-(\beta - \rho)t}\PP_{i\beta}[g](\gamma_\xi(t)) \le \frac{C}{\tau^{2\beta}} \:. 
\end{equation*}
Thus by (\ref{approx_identity}),
\begin{equation*}
e^{-(\beta - \rho)t}\PP_{i\beta}[f](\gamma_\xi(t)) = 1 - e^{-(\beta - \rho)t}\PP_{i\beta}[g](\gamma_\xi(t)) \ge 1 - \frac{C}{\tau^{2\beta}} \:. 
\end{equation*}
\end{proof}

\subsection{Upper bound on the Hausdorff dimension}
\label{subsec4.2}
\begin{proof}[Proof of Theorem \ref{exceptional_thm}]
For $L>0$, we set
\begin{equation} \label{defn_EL}
E^L_{\alpha}(\PP_{i\beta}[\mu]) :=\left\{\xi \in \partial \X : \displaystyle\limsup_{t \to +\infty} e^{-(\alpha + \beta -\rho) t} \left|\PP_{i\beta}[\mu]\left(\gamma_{\xi}(t)\right)\right| > L\right\} \:.
\end{equation}

\medskip

Our strategy will be to get some useful estimates on the $(h-\alpha)$\:-dimensional outer Hausdorff measure of the set defined in (\ref{defn_EL}). First we choose and fix $\varepsilon \in (0,1)$ and $\xi \in E^L_{\alpha}(\PP_{i\beta}[\mu])$. Then by Corollary \ref{cor1} there exists $C > 0$, depending only on $\beta$ and the intrinsic geometry of $\X$, such that
\begin{equation*}
CL < \displaystyle\limsup_{t \to +\infty} e^{-\alpha t}\: \M_{e^{-t}, \varepsilon}[\mu](\xi).
\end{equation*}
Hence, there exists $t_\xi \in (0,+\infty)$ satisfying $e^{-t_{\xi}} \le \varepsilon$ such that
\begin{equation} \label{poisson_pf_eq}
CL < e^{-\alpha t_\xi}\: \frac{|\mu|\left(\mathscr{B}\left(\xi, e^{-t_\xi}\right)\right)}{e^{-ht_\xi}} \:.
\end{equation}
Now by Vitali 5-covering Lemma for quasi-metric spaces, there exist countably many visual balls $\{\mathscr{B}\left(\xi_j, r_{\xi_j}\right)\}_{j=1}^\infty$ satisfying (\ref{poisson_pf_eq}) such that
\begin{itemize}
\item $r_{\xi_j} := e^{-t_{\xi_j}} \le \varepsilon$\:, for all $j \in \N$,
\item $\mathscr{B}\left(\xi_j, r_{\xi_j}\right) \cap \mathscr{B}\left(\xi_k, r_{\xi_k}\right) = \emptyset$ for all $j \ne k$,
\item $E^L_{\alpha}(\PP_{i\beta}[\mu]) \subset \displaystyle\bigcup_{j=1}^\infty \mathscr{B}\left(\xi_j, 5e^{4\delta}r_{\xi_j}\right)$ \:.
\end{itemize} 

Then by (\ref{poisson_pf_eq}), 
\begin{eqnarray*}
\displaystyle\sum_{j=1}^\infty {\left(\mathfrak{d}\left(\mathscr{B}\left(\xi_j, 5e^{4\delta}r_{\xi_j}\right)\right)\right)}^{h-\alpha}  & \lesssim & \left(\frac{1}{L}\right) \displaystyle\sum_{j=1}^\infty |\mu|\left(\mathscr{B}\left(\xi_j, r_{\xi_j}\right)\right) \\
&=&  \left(\frac{1}{L}\right)  |\mu|\left(\bigcup_{j=1}^\infty \mathscr{B}\left(\xi_j, r_{\xi_j}\right)\right) \\
&\le &  \left(\frac{1}{L}\right)  |\mu|(\partial \X) \:.
\end{eqnarray*}
We note that the implicit constant above is independent of the choice of $\varepsilon$ and hence letting $\varepsilon \to 0$, we get that 
\begin{equation} \label{poisson_pf_eq2}
\mathcal{H}^{h-\alpha}\left(E^L_{\alpha}(\PP_{i\beta}[\mu])\right) \lesssim \left(\frac{1}{L}\right) |\mu|\left(\partial \X\right) < +\infty \:. 
\end{equation}

As $E^\infty_{\alpha}(\PP_{i\beta}[\mu]) \subset E^L_{\alpha}(\PP_{i\beta}[\mu])$ for all $L>0$, it follows that
\begin{equation*}
\mathcal{H}^{h-\alpha}\left(E^\infty_{\alpha}(\PP_{i\beta}[\mu])\right) =0 \:.
\end{equation*}
Moreover, (\ref{poisson_pf_eq2}) implies that
\begin{equation*}
dim_{\mathcal{H}} E^{\frac{1}{m}}_\alpha (\PP_{i\beta}[\mu]) \le h-\alpha\:,\: \forall m \in \N\:.
\end{equation*}
Finally combining countable stability of the Hausdorff dimension and (\ref{poisson_pf_eq2}) we obtain,
\begin{equation*}
dim_{\mathcal{H}} E_{\alpha}(\PP_{i\beta}[\mu]) = \displaystyle\sup_{m \in \N} \left\{dim_{\mathcal{H}} E^{\frac{1}{m}}_\alpha (\PP_{i\beta}[\mu])\right\} \le h-\alpha \:.
\end{equation*}
\end{proof}

\subsection{Sharpness of the size estimates}
\label{subsec4.3}
\begin{proof}[Proof of Theorem \ref{exceptional_sharp_thm}]
We first prove part (i). Since $\mathcal{H}^{h-\alpha}(E)=0$, for any $m \in \N$, there exists a covering of $E$ by visual balls $\{\mathscr{B}^{(m,j)}\}_{j=1}^\infty$ such that their diameters satisfy
\begin{equation} \label{poisson_sharp_eq1}
\displaystyle\sum_{j=1}^\infty {\left(\mathfrak{d}\left(\mathscr{B}^{(m,j)}\right)\right)}^{h-\alpha} < 2^{-m} \:.
\end{equation}

If $\mathscr{B}$ is a visual ball with center $\eta \in \partial \X$ and with radius $r$, then for notational convenience, $2e^{2\delta}\mathscr{B}$ will denote the visual ball with the same center $\eta$ and radius $2e^{2\delta}r$.

We now define, 
\begin{equation*} 
f:= \displaystyle\sum_{j,m} m\:{\left(\mathfrak{d}\left(\mathscr{B}^{(m,j)}\right)\right)}^{-\alpha}\: \chi_{2e^{2\delta}\mathscr{B}^{(m,j)}} \:.
\end{equation*}

Then by the estimate on the visibility measure on visual balls (\ref{upper_bound_visual_measure}) and the choice (\ref{poisson_sharp_eq1}), it follows that 
\begin{eqnarray*}
\int_{\partial \X} f\: d \lambda_o & \le & \displaystyle\sum_{j,m} m \: {\left(\mathfrak{d}\left(\mathscr{B}^{(m,j)}\right)\right)}^{-\alpha} \:\lambda_o\left(2e^{2\delta}\mathscr{B}^{(m,j)}\right) \\
& \lesssim & \displaystyle\sum_{j,m} m \:{\left(\mathfrak{d}\left(\mathscr{B}^{(m,j)}\right)\right)}^{h-\alpha} \\
& < &   \:\displaystyle\sum_{m=1}^\infty \frac{m}{2^m} \\
& < & +\infty \:.
\end{eqnarray*}

Thus $f d\lambda_o$ defines a Radon measure on $\partial \X$. Now let $\xi \in E$ and fix $m \in \N$. Then there exists $j_m \in \N$ such that $\xi \in \mathscr{B}^{(m,j_m)}$. If $r_m$ is the radius of $\mathscr{B}^{(m,j_m)}$, then $\mathscr{B}(\xi,r_m) \subset 2e^{2\delta}\mathscr{B}^{(m,j_m)}$. Then by Corollary \ref{cor2}, one has
\begin{equation}\label{poisson_sharp_eq2}
e^{-(\beta - \rho)t}\PP_{i\beta}\left[\chi_{2e^{2\delta}\mathscr{B}^{(m,j_m)}}\right](\gamma_\xi(t))
 \ge e^{-(\beta - \rho)t}\PP_{i\beta}\left[\chi_{\mathscr{B}(\xi,r_m)}\right](\gamma_\xi(t)) \ge   \frac{1}{2} \:,  
\end{equation}
whenever (following the statement of Corollary \ref{cor2})
\begin{itemize}
\item $\tau^{2\beta} > \max \left\{1,\: 2C_1\right\}$,
\item $t > \log(\tau)$,
\item $\tau e^{-t} \le r_m$\:.
\end{itemize}

Hence choosing $\tau >0$ sufficiently large and setting 
\begin{equation*}
t_m := \log\left(\frac{\tau}{r_m}\right)  \:,
\end{equation*}
we have by (\ref{poisson_sharp_eq2}),
\begin{eqnarray*}
e^{-(\beta - \rho)t_m}\PP_{i\beta}\left[f\right](\gamma_\xi(t_m)) & \ge & m {\left(\mathfrak{d}\left(\mathscr{B}^{(m,j_m)}\right)\right)}^{-\alpha} e^{-(\beta - \rho)t_m}\PP_{i\beta}\left[\chi_{2e^{2\delta}\mathscr{B}^{(m,j_m)}}\right](\gamma_\xi(t_m)) \\
& \ge & m \left(2^{-\left(\alpha +1\right)}\:e^{-2\alpha \delta}\: \tau^{-\alpha} \right)  e^{\alpha t_m} \:. 
\end{eqnarray*}

Hence, there exists $C >0$, depending only on $\alpha,\beta$ and the intrinsic geometry of $\X$, such that for all $m \in \N$,
\begin{equation} \label{poisson_sharp_eq3}
e^{-(\alpha + \beta - \rho)t_m}\PP_{i\beta}\left[f\right](\gamma_\xi(t_m)) \ge C\: m \:.
\end{equation}

Now by (\ref{poisson_sharp_eq1}),
\begin{equation*}
t_m > \log\left(\tau\right) +  m\left(\frac{\log(2)}{h-\alpha}\right) \to +\infty \text{ as } m \to +\infty \:.
\end{equation*}

Hence (\ref{poisson_sharp_eq3}) gives part (i) of the result.

\medskip

For part (ii), we start off by considering $\alpha=0$. In this case, simply taking $\mu=\lambda_o$, we note that by (\ref{approx_identity}),
\begin{equation*}
e^{-(\beta - \rho)d(o,x)} \PP_{i\beta}[\lambda_o](x) \equiv 1\:,\:\:x \in \X\:,
\end{equation*}
and thus 
\begin{equation*}
E_0\left(\PP_{i\beta}[\lambda_o]\right)=\partial \X\:.
\end{equation*}
Thus by Theorem \ref{exceptional_thm}, we have that 
\begin{equation*}
dim_{\mathcal{H}} \left(\partial \X\right)= dim_{\mathcal{H}} \left(E_0\left(\PP_{i\beta}[\lambda_o]\right)\right) \le h\:.
\end{equation*}
To get the lower bound on the Hausdorff dimension, we recall that $\nu$ only defines a quasi-metric on $\partial \X$. Let $-s^2_0$ be the `asymptotic upper curvature bound of $\X$', where $s_0 \in (0,+\infty]$ is the critical exponent such that for all $s \in (0,s_0)$, there exists a metric $\nu_s$ which is bi-Lipschitz to $\nu^s$ (see \cite{BF, S}). For any such $s \in (0,s_0)$, we thus have the containment of metric balls
\begin{equation*}
\B_{\nu_s}\left(\xi,r\right) \subset \B\left(\xi, Cr^{\frac{1}{s}}\right)\:,\:\:\xi \in \partial \X,\: r \in (0,1)\:,
\end{equation*}
where $C \ge 1$ is some constant depending only on the visual parameter $s$. Thus by the estimate (\ref{upper_bound_visual_measure}), we get that
\begin{equation*}
\lambda_o\left(\overline{\B_{\nu_s}\left(\xi,r\right)}\right) \lesssim r^{\frac{h}{s}}\:, \:\:\xi \in \partial \X,\: r \in (0,1)\:,
\end{equation*}  
with the implicit constant only depending on the intrinsic geometry of $\X$ and the visual parameter. Hence the Hausdorff dimension of $\partial \X$ with respect to the metric $\nu_s$ is $\ge h/s$ (see \cite[p.61]{Hei}). Then by power scaling and and bi-Lipschitz invariance of Hausdorff dimension \cite[Section 2.2]{Falconer}, it follows that for the quasi-metric $\nu$,  $dim_{\mathcal{H}} \left(\partial \X\right) \ge h$ and hence equal to $h$. This gives the result for $\alpha=0$ and also proves Corollary \ref{rigidity_cor}.

\medskip

We next choose and fix $\alpha \in (0,h)$. Again considering the metric $\nu_s$, with visual parameter $s \in (0,s_0)$, we note that (by the above argument) the Hausdorff dimension of $\partial \X$ with respect to $\nu_s$ is equal to $h/s$. Then by Lemma \ref{existence}, for each integer $j > 1/(h-\alpha)$, there exists a compact subset $K_j \subset \partial \X$ with Hausdorff dimension equal to $\left(h-\alpha-\frac{1}{j}\right)/s$, with respect to $\nu_s$. Hence by power scaling and bi-Lipschitz invariance of Hausdorff dimension, we get that for the quasi-metric $\nu$,
\begin{equation} \label{partii_eq1}
dim_{\mathcal{H}}\left(K_j\right)=h-\alpha-\frac{1}{j}\:.
\end{equation}
We now consider
\begin{equation*}
K:= \bigcup_j K_j\:.
\end{equation*}
Then for $\tau \ge h - \alpha$, by (\ref{partii_eq1}),
\begin{equation} \label{partii_eq2}
\mathcal{H}^\tau(K) \le \sum_j \mathcal{H}^\tau(K_j) =0\:.
\end{equation}
On the other hand, for $\tau < h - \alpha$, there exists $j_0 \in \N$, such that
\begin{equation*}
\tau < h-\alpha-\frac{1}{j_0}\:,
\end{equation*}
and hence by (\ref{partii_eq1}),
\begin{equation} \label{partii_eq3}
\mathcal{H}^\tau(K) \ge \mathcal{H}^\tau(K_{j_0}) = \infty\:.
\end{equation}
Thus by (\ref{partii_eq2}) and (\ref{partii_eq3}), it follows that
\begin{equation} \label{partii_eq4}
dim_{\mathcal{H}}\left(K\right)= h- \alpha\:,\text{ with } \mathcal{H}^{h-\alpha}(K)=0\:.
\end{equation}
Then by part (i) of Theorem \ref{exceptional_sharp_thm},  there exists a Radon measure $\mu$ on $\partial \X$, such that $K \subset E^\infty_\alpha(\PP_{i\beta}[\mu])\:.$ As by definition, $E^\infty_\alpha(\PP_{i\beta}[\mu]) \subset E_\alpha(\PP_{i\beta}[\mu])$, we then have $K \subset  E_\alpha(\PP_{i\beta}[\mu])$\:. Thus combining (\ref{partii_eq4}) and Theorem \ref{exceptional_thm}, we get that
\begin{equation*}
h- \alpha = dim_{\mathcal{H}}\left(K\right) \le dim_{\mathcal{H}}\left(E_\alpha(\PP_{i\beta}[\mu])\right) \le h- \alpha\:.  
\end{equation*}
This completes the proof of Theorem \ref{exceptional_sharp_thm}.
\end{proof}

\section*{Part II: Boundary behaviour of superharmonic functions}

\section{Non-tangential limits of Green potentials}
\label{sec5}
In this section, we study the non-tangential boundary behaviour of Green potentials, the genuine non-harmonic part of a positive superharmonic function and work under the hypothesis of Theorem \ref{non-tangential_superharmonic_thm}. 

We start off with the following result on the growth of a Radon measure which is equivalent to the existence of its Green potential:
\begin{lemma} \label{Green_potential_lemma}
For a  Radon measure $\mu$ on $\X$, the Green potential $G[\mu]$ is well-defined if and only if 
\begin{equation} \label{Green_potential_condition}
\int_{\X} e^{-hd(o,x)}\:d\mu(x) <\infty\:.
\end{equation}
\end{lemma}
\begin{proof}
The key to this result is the estimate of the Green function (\ref{green_estimate}). To prove the necessity, we assume that $G[\mu]$ is well-defined, that is, there exists $x_0 \in \X$ such that $G[\mu](x_0)<\infty$. Then, we first decompose,
\begin{eqnarray*}
\int_{\X} e^{-hd(o,x)}\:d\mu(x) = \int_{B(o,1)} e^{-hd(o,x)}\:d\mu(x) \:+\: \int_{\X \setminus B(o,1)} e^{-hd(o,x)}\:d\mu(x) \:.
\end{eqnarray*}
The first integral on the right hand side is finite as $\mu$ is Radon. To estimate the second integral, we note by the triangle inequality and the estimate (\ref{green_estimate}) that
\begin{equation*}
\int_{\X \setminus B(o,1)} e^{-hd(o,x)}\:d\mu(x) \le e^{hd(o,x_0)}  \int_{\X \setminus B(o,1)} e^{-hd(x_0,x)}\:d\mu(x) \lesssim e^{hd(o,x_0)}  G[\mu](x_0)<\infty\:.
\end{equation*}

Conversely, to prove the sufficiency, let us assume the validity of (\ref{Green_potential_condition}). Then for any $x \in \X$, we decompose,
\begin{eqnarray*}
G[\mu](x)&=& \int_{B(x,1)} G(x,y)\:d\mu(y) \:+\: \int_{\X \setminus B(x,1)} G(x,y)\:d\mu(y)\\
&=:& u_1(x)\:+\:u_2(x)\:.
\end{eqnarray*}
To bound $u_2(x)$, we again apply the estimate (\ref{green_estimate}), the triangle inequality and the hypothesis (\ref{Green_potential_condition}), to get
\begin{equation*}
u_2(x) \lesssim \int_{\X \setminus B(x,1)} e^{-hd(x,y)}\:d\mu(y) \le e^{hd(o,x)} \int_{\X \setminus B(x,1)} e^{-hd(o,y)}\:d\mu(y) <\infty\:.
\end{equation*}
To complete the proof, we will show that $u_1$ is finite almost everywhere on $\X$. Thus it suffices to show that on geodesic balls $B(o,R)$, for any $R>0$, $u_1 \in L^1(B(o,R))$. Indeed, by Fubini-Tonelli's theorem, local integrability, radiality of the Green function and the assumption that $\mu$ is Radon, we get
\begin{eqnarray*}
\int_{B(o,R)} u_1(x)\:dvol(x) &=& \int_{B(o,R)}  \int_{B(x,1)} G(x,y)\:d\mu(y)\:dvol(x) \\
&\le &  \int_{B(o,R+1)}  \left(\int_{B(y,1)} G(x,y)\:dvol(x)\right)\:d\mu(y) \\
& \lesssim & \mu\left(B(o,R+1)\right) < \infty\:.
\end{eqnarray*}
\end{proof}

The main result of this section is the following:
\begin{theorem} \label{non-tangential_green_thm}
Let $\X,\:\partial \X,\:h$ be as in Theorem \ref{non-tangential_superharmonic_thm}. Let $\psi$ be a non-negative measurable function on $\X$ such that $\psi \:dvol$ is a Radon measure on $\X$ and $G[\psi]$ is well-defined. If 
\begin{equation} \label{Lp_condition}
\int_{\X} e^{-hd(o,x)}\:\psi(x)^p\:dvol(x) <\infty\:,\:\:\text{ for some } p>\frac{n}{2}\:,
\end{equation} 
then $G[\psi]$ has non-tangential limit $0$ at $\lambda_o$-a.e. point on $\partial \X$.
\end{theorem}
We first note that in light of Lemma \ref{Green_potential_lemma}, the well-definedness of the Green potential $G[\psi]$ is equivalent to the condition:
\begin{equation} \label{master_condition}
\int_{\X} e^{-hd(o,x)}\:\psi(x)\:dvol(x)<\infty\:.
\end{equation}
To prove Theorem \ref{non-tangential_green_thm}, we again proceed as in the proof of Lemma \ref{Green_potential_lemma}, by decomposing the Green potential into polar and non-polar parts:
\begin{equation*}
G[\psi](x)=u_1(x)\:+\:u_2(x)\:,\:\:x \in \X\:,
\end{equation*}
where
\begin{eqnarray*}
&&u_1(x):= \int_{B(x,1)} G(x,y)\:\psi(y)\:dvol(y)\:,\\
&&u_2(x):= \int_{\X \setminus B(x,1)} G(x,y)\:\psi(y)\:dvol(y)\:.
\end{eqnarray*}
Then Theorem \ref{non-tangential_green_thm} follows from the following Lemmata:
\begin{lemma} \label{non-polar_lemma}
The function $u_2$ has non-tangential limit $0$ at $\lambda_o$-a.e. point on $\partial \X$.
\end{lemma}
\begin{lemma} \label{polar_lemma}
The function $u_1$ has non-tangential limit $0$ at $\lambda_o$-a.e. point on $\partial \X$.
\end{lemma}
We prove Lemmata \ref{non-polar_lemma} and \ref{polar_lemma} in Subsections \ref{subsec5.1} and \ref{subsec5.2} respectively. 
\subsection{Non-polar part} \label{subsec5.1} 
\begin{proof}[Proof of Lemma \ref{non-polar_lemma}]
The key idea is to control the non-tangential limit of $u_2$ in terms of some suitable Poisson integrals.

Let $x \in \X$. For $y \in \X \setminus B(x,1)$, we have by the pointwise decay of the Green function away from the pole, given by (\ref{green_estimate}),
\begin{equation*}
u_2(x) \lesssim \int_{\X \setminus B(x,1)} e^{-hd(x,y)}\:\psi(y)\:dvol(y) \le \int_{\X} e^{-hd(x,y)}\:\psi(y)\:dvol(y)\:.
\end{equation*}
Fixing an $R>0$, we decompose, $$\X = B(o,R) \sqcup \left(\X \setminus B(o,R)\right)\:.$$
On $B(o,R)$, by triangle inequality and the fact that $\psi\:dvol$ is a Radon measure, we have
\begin{equation} \label{non-polar_eq1}
\int_{B(o,R)} e^{-hd(x,y)}\:\psi(y)\:dvol(y) \le \left( e^{hR}  \|\psi\|_{L^1(B(o,R))}\right) e^{-hd(o,x)}=:C_R\:e^{-hd(o,x)}\:.
\end{equation}
On the other hand, on $\X \setminus B(o,R)$,
\begin{equation*}
\int_{\X \setminus B(o,R)} e^{-hd(x,y)}\:\psi(y)\:dvol(y) = \int_{\X \setminus B(o,R)} e^{-hB_y(x)}\:e^{-hd(o,y)}\:\psi(y)\:dvol(y)\:,
\end{equation*}
where
\begin{equation*}
B_y(x):= d(x,y)-d(o,y)\:,\:\:x,y \in \X\:.
\end{equation*}
Now rewriting the above in terms of the Gromov product, we have
\begin{equation*}
B_y(x)= d(o,x)-2(x|y)_o\:.
\end{equation*}
Consider the geodesic segment that joins $o$ to $y$ and extend it to $\partial \X$. There exists a unique point on $\partial \X$ where this extended infinite geodesic ray, say $\gamma$, meets $\partial \X$, say $\eta_y$. Combining this with the fact that the Gromov product is monotonically increasing along geodesic rays, it follows that $y \mapsto B_y(x)$ is monotonically decreasing along $\gamma$. Hence we get, 
\begin{equation*}
\int_{\X \setminus B(o,R)} e^{-hd(x,y)}\:\psi(y)\:dvol(y) \le \int_{\X \setminus B(o,R)} e^{-hB_{\eta_y}(x)}\:e^{-hd(o,y)}\:\psi(y)\:dvol(y)\:.
\end{equation*}
Now by the hypothesis (\ref{master_condition}) and the Riesz representation theorem, there exists a Radon measure $\omega_R$ on $\partial \X$ such that for all continuous functions $\phi$ on $\partial \X$, we have
\begin{equation*}
\int_{\X \setminus B(o,R)} \phi(\eta_y)\:e^{-hd(o,y)}\:\psi(y)\:dvol(y)= \int_{\partial \X} \phi(\eta)\:d\omega_R(\eta)\:.
\end{equation*} 
Thus by the boundary continuity of the Busemann function, we have
\begin{equation} \label{non-polar_eq2}
\int_{\X \setminus B(o,R)} e^{-hd(x,y)}\:\psi(y)\:dvol(y) \le \int_{\partial \X} e^{-hB_\eta(x)}\:d\omega_R(\eta)=\PP\left[\omega_R\right](x)\:.
\end{equation}
Hence combining (\ref{non-polar_eq1}) and (\ref{non-polar_eq2}), we get
\begin{equation*}
u_2(x) \lesssim C_R\:e^{-hd(o,x)}\:+\:\PP\left[\omega_R\right](x)\:,
\end{equation*}
and thus for $\alpha >1,\:\xi \in \partial \X$, by Lemma \ref{maximal_estimate_lemma},
\begin{eqnarray*}
\limsup_{\substack{x \to \xi \\ x \in \Gamma_\alpha(\xi)}} u_2(x) \lesssim  \limsup_{\substack{x \to \xi \\ x \in \Gamma_\alpha(\xi)}}  \PP\left[\omega_R\right](x) \le \M_\alpha \PP\left[\omega_R\right](\xi) \lesssim \M_{HL} \omega_R(\xi)\:,
\end{eqnarray*}
where the implicit constants depend only on $\alpha$ and the intrinsic geometry of $\X$. Thus by Lemma \ref{Sec3_lemma1}, for any $t>0$,
\begin{equation} \label{non-polar_eq3}
\lambda_o(\mathcal{E}_t) \le C\frac{\omega_R(\partial \X)}{t}\:,
\end{equation}
where $C$ only depends on $\alpha$ and the intrinsic geometry of $\X$ and 
\begin{equation*}
\mathcal{E}_t := \left\{\xi \in \partial \X : \limsup_{\substack{x \to \xi \\ x \in \Gamma_\alpha(\xi)}} u_2(x) > t\right\}\:.
\end{equation*} 
Suppose \begin{equation*}
\mathcal{E} := \left\{\xi \in \partial \X : \limsup_{\substack{x \to \xi \\ x \in \Gamma_\alpha(\xi)}} u_2(x) > 0\right\}\:,
\end{equation*}
has positive $\lambda_o$ measure. Then we note that as $\mathcal{E} \subset \displaystyle\cup_{n =1}^\infty \mathcal{E}_{\frac{1}{n}}$, there must exist some $n \in \N$, such that $\lambda_o\left(\mathcal{E}_{\frac{1}{n}}\right)>0$. For this $n \in \N$, we choose $R>0$ such that
\begin{equation*}
C\omega_R\left(\partial \X\right) < \frac{\lambda_o(\mathcal{E}_{\frac{1}{n}})}{2n}\:,
\end{equation*}
contradicting (\ref{non-polar_eq3}) for $t=1/n$. This completes the proof.
\end{proof}
\begin{remark} \label{non-polar_remark}
We note that the proof of Lemma \ref{non-polar_lemma} works for general Radon measures $\mu$ such that $G[\mu]$ is well-defined. Thus for the a.e. existence of the non-tangential boundary limits of the non-polar part, the absolute continuity of the Riesz measure and the suitable weighted integrability condition on the density imposed by (\ref{density_condition}) in Theorem \ref{non-tangential_superharmonic_thm} are superflous. 
\end{remark}

\subsection{Polar part} \label{subsec5.2}
In this subsection, we prove Lemma \ref{polar_lemma}. We start off with the following definition:
\begin{definition} \label{non-tangential-shadow}
For $x \in \X$ and $\alpha >1$, we set
\begin{equation*}
\tilde{\Gamma}_\alpha(x):=\left\{\xi \in \partial \X : x \in \Gamma_\alpha(\xi)\right\}\:.
\end{equation*}
The set defined above is on $\partial \X$ and is illustrated by the red region in Figure \ref{figure1}.
\end{definition}
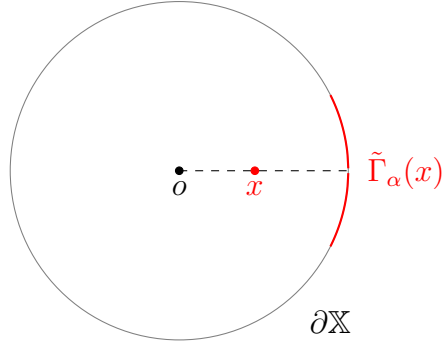
\begin{figure}[h]
\begin{center}
\begin{tikzpicture}
	\draw [gray] (4,4) circle [radius={sqrt(5)}];
	\draw [fill] (4,4) circle [radius=.05] node [below] {$o$};
	\draw [dashed] (4,4) -- (6.236,4);
	\draw [red][fill] (5,4) circle [radius=.05] node [below] {$x$};
	\draw[red][thick] plot[domain=6:6.236, smooth] (\x,{4+sqrt(5-(\x-4)^2)});
	\draw[red][thick] plot[domain=6:6.236, smooth] (\x,{4-sqrt(5-(\x-4)^2)});
	\pgftext[at={\pgfpoint{7cm}{4cm}}]{\textcolor{red}{$\tilde{\Gamma}_\alpha(x)$}};
	\pgftext[at={\pgfpoint{6cm}{2cm}}]{$\partial \X$};
	  \end{tikzpicture}
\end{center}
\caption{The set $\tilde{\Gamma}_\alpha(x)$}
\label{figure1}
\end{figure}

The following result can be viewed as a `shadow lemma' which illustrates a containment relation between $\tilde{\Gamma}_\alpha(x)$ and the visual balls centred at the projection of $x$:
\begin{lemma} \label{shadow_lemma1}
For $x \in \X$, let $\eta_x \in \partial \X$ be the projection of $x$ from $o$ to $\partial \X$. Then $\tilde{\Gamma}_\alpha(x) \subset \B\left(\eta_x,\:\alpha\:e^{-d(o,x)}\right)$, for all $\alpha>1$.
\end{lemma}
\begin{proof}
We note that $\xi \in \tilde{\Gamma}_\alpha(x)$ implies that 
\begin{equation*}
e^{-(x|\xi)_o} < \alpha e^{-d(o,x)}\:.
\end{equation*}
Now as Gromov products are increasing along geodesic rays, we have $(x|\xi)_o \le (\eta_x|\xi)_o$, and hence
\begin{equation*}
\nu(\xi,\eta_x)=e^{-(\eta_x|\xi)_o} \le e^{-(x|\xi)_o} < \alpha e^{-d(o,x)}\:.
\end{equation*}
This completes the proof. 
\end{proof}

Next, for a Radon measure $\mu$ on $\X$, $\alpha>1$ and $\xi \in \partial \X$, we define the mass operator of $\mu$ on non-tangential cones $\Gamma_\alpha(\xi)$,
\begin{equation} \label{mass_non-tangential_cone}
\Ss^*_\alpha\mu(\xi):= \mu\left(\Gamma_\alpha(\xi)\right)\:.
\end{equation}
For $R>0$, let $A_R:= \X \setminus B(o,R)$. Our next result shows an interesting property of the operator (\ref{mass_non-tangential_cone}):
\begin{lemma} \label{mass_lemma}
If a Radon measure $\mu$ on $\X$ satisfies the condition (\ref{Green_potential_condition}), then for all $\alpha>1$,
\begin{itemize}
\item[(i)] $\Ss^*_\alpha\mu \in L^1(\partial \X)\:;$
\item[(ii)] $\mu\left(\Gamma_\alpha(\xi)\right)< \infty$, for $\lambda_o$-a.e. $\xi \in \partial \X$, and thus
\begin{equation*}
\lim_{R \to \infty} \mu\left(\Gamma_\alpha(\xi) \cap A_R\right)=0\:,\:\:\lambda_o\text{-a.e. } \xi \in \partial \X\:.
\end{equation*}
\end{itemize}
\end{lemma}
\begin{proof}
Let $\mu$ satisfy (\ref{Green_potential_condition}). Now By Fubini-Tonelli's theorem,
\begin{eqnarray*}
\int_{\partial \X} \Ss^*_\alpha\mu(\xi)\:d\lambda_o(\xi) &=& \int_{\partial \X} \int_\X \chi_{\Gamma_\alpha(\xi)}(x)\:d\mu(x) \:d\lambda_o(\xi)\\
&=& \int_\X \left(\int_{\partial \X} \chi_{\tilde\Gamma_\alpha(x)}(\xi)\:d\lambda_o(\xi)\right)\:d\mu(x) \\
&=& \int_\X  \lambda_o\left(\tilde\Gamma_\alpha(x)\right)\:d\mu(x)\:.
\end{eqnarray*}
Thus by Lemma \ref{shadow_lemma1} and the estimate of the visibility measure of visual balls given by (\ref{upper_bound_visual_measure}), we get
\begin{equation*}
\int_{\partial \X} \Ss^*_\alpha\mu(\xi)\:d\lambda_o(\xi) \le \int_\X  \lambda_o\left(\B\left(\eta_x,\:\alpha\:e^{-d(o,x)}\right)\right)\:d\mu(x) \lesssim \int_\X e^{-hd(o,x)}\:d\mu(x)\:,
\end{equation*}
where the implicit constant only depends on $\alpha$ and the intrinsic geometry of $\X$. Now by (\ref{Green_potential_condition}), the last integral is finite and hence $\Ss^*_\alpha\mu \in L^1(\partial \X)$. This completes the proof of part (i).

The part (ii) of the result immediately follows from part (i).
\end{proof}

We are now all set to prove Lemma \ref{polar_lemma}.
\begin{proof}[Proof of Lemma \ref{polar_lemma}]
We recall that
\begin{equation*}
u_1(x):= \int_{B(x,1)} G(x,y)\:\psi(y)\:dvol(y)\:,\:\:x \in \X\:.
\end{equation*}
By H\"older's inequality,
\begin{equation*}
u_1(x) \le \left(\int_{B(x,1)} \psi(y)^p\:dvol(y)\right)^\frac{1}{p}\left(\int_{B(x,1)} G(x,y)^q\:dvol(y)\right)^\frac{1}{q}\:,
\end{equation*}
where $p>n/2$, as in the hypothesis of Theorem \ref{non-tangential_green_thm} and $q$ its H\"older conjugate. Thus $q<n/(n-2)$. Now by the radiality and pointwise estimates of the Green function near the pole (\ref{green_estimate}) and the local growth of the volume density function (\ref{jacobian_estimate}),
\begin{equation*}
\int_{B(x,1)} G(x,y)^q\:dvol(y) \lesssim \int_0^1 r^{n-1-q(n-2)}\:dr\:.
\end{equation*}
Then as noted above, $q<n/(n-2)$ and hence the last integral is finite. Therefore,
\begin{equation*}
u_1(x) \lesssim \left(\int_{B(x,1)} \psi(y)^p\:dvol(y)\right)^\frac{1}{p}\:.
\end{equation*}
Now let $\xi \in \partial \X,\:\alpha >1$ and $R>0$. Then for all $x \in \Gamma_\alpha(\xi) \cap A_{R+1}$, we note that $B(x,1) \subset \Gamma_{e^2\alpha}(\xi) \cap A_R$. Indeed, for any $y \in B(x,1)$, by triangle inequality,
\begin{equation*}
d(o,y) \ge d(o,x)-d(x,y) > (R+1)-1=R\:.
\end{equation*}
Also, as 
\begin{equation*}
(x|\xi)_o \le (y|\xi)_o \:+\: d(x,y) < (y|\xi)_o \:+\: 1\:,
\end{equation*}
we have,
\begin{equation*}
e^{-(y|\xi)_o} < e\:e^{-(x|\xi)_o} < e\:\alpha\:e^{-d(o,x)} \le e^2\alpha\:e^{-d(o,y)}\:.
\end{equation*}
Thus for all $x \in \Gamma_\alpha(\xi) \cap A_{R+1}$,
\begin{equation*}
u_1(x) \lesssim \left(\int_{\Gamma_{e^2\alpha}(\xi) \cap A_R} \psi(y)^p\:dvol(y)\right)^\frac{1}{p}\:.
\end{equation*}
Now under the hypothesis (\ref{Lp_condition}), Lemma \ref{mass_lemma} applies to $d\mu=\psi^p\:dvol$. The result then follows by part (ii) of Lemma \ref{mass_lemma}. This also completes the proof of Theorem \ref{non-tangential_green_thm}.
\end{proof}

\section{Tangential limits of Green potentials and exceptional sets}
\label{sec6}
In this section, we prove:
\begin{theorem} \label{tangential_limit_green_thm}
Let $\X$ and $h$ be as in Theorem \ref{tangential_limit_superharmonic_thm}. Let $\psi$ be a non-negative measurable function on $\X$ such that $\psi\:dvol$ is a Radon measure on $\X$ and $G[\psi]$ is well-defined. If
\begin{equation} \label{Lp_condition'}
\int_{\X} e^{-\beta d(o,x)}\:\psi(x)^p\:dvol(x) < \infty\:,\:\: \text{ for some }\beta \in (0,h),\:\:p>\frac{n}{2}\:,
\end{equation}
then for each $\tau \in (1,h/\beta)$, there exists a set $E \subset \partial \X$ with $\mathcal{H}^{\beta \tau}(E)=0$ such that $G[\psi]$ has tangential limit $0$ of degree $\tau$ at all $\xi \in \partial \X \setminus E\:.$ 
\end{theorem}

Before arriving at the proof of Theorem \ref{tangential_limit_green_thm}, we will require some preparation. Henceforth, in this section, we work under the hypothesis of Theorem \ref{tangential_limit_green_thm}.

\begin{lemma} \label{energy_estimates}
For $\alpha>h,\:t>0,\: x \in \X\:,$ we have the estimate
\begin{equation*}
\int_{\X} e^{(\alpha +t)(x|y)_o}\:e^{-\alpha d(o,y)}\:dvol(y) \lesssim e^{td(o,x)}\:,
\end{equation*}
where the implicit constant only depends on the parameters $\alpha,\:t$ and the intrinsic geometry of $\X$.
\end{lemma}
\begin{proof}
We fix $x \in \X$. By polar decomposition of the volume measure, we have
\begin{equation*}
\int_{\X} e^{(\alpha +t)(x|y)_o}\:e^{-\alpha d(o,y)}\:dvol(y) \lesssim \int_{T^1_o\X}\int_0^\infty e^{(\alpha +t)\left(x|\gamma_v(r)\right)_o}\:e^{(h-\alpha) r}\:dr\:d\theta_o(v)\:.
\end{equation*}
Consider the geodesic segment $\gamma_v$ that joins $o$ to $\gamma_v(r)$ and extend it to $\partial \X$. This extended infinite geodesic ray meets $\partial \X$ at a unique point, say $\gamma_v(\infty)$. Then by (\ref{four_pt_condition'}), we note that
\begin{eqnarray} \label{energy_estimates_eq1}
\left(x|\gamma_v(r)\right)_o &\ge & \min\left\{\left(x|\gamma_v(\infty)\right)_o,\left(\gamma_v(r)|\gamma_v(\infty)\right)_o\right\}-2\delta \nonumber\\
&=& \min\left\{\left(x|\gamma_v(\infty)\right)_o,r\right\}-2\delta\:.
\end{eqnarray}
This prompts us to decompose the last integral as,
\begin{eqnarray*}
&&\int_{T^1_o\X}\int_0^\infty e^{(\alpha +t)\left(x|\gamma_v(r)\right)_o}\:e^{(h-\alpha) r}\:dr\:d\theta_o(v)\\
&=& \int_{T^1_o\X}\int_0^{\left(x|\gamma_v(\infty)\right)_o} e^{(\alpha +t)\left(x|\gamma_v(r)\right)_o}\:e^{(h-\alpha) r}\:dr\:d\theta_o(v) \\
&&+ \int_{T^1_o\X}\int_{\left(x|\gamma_v(\infty)\right)_o}^\infty e^{(\alpha +t)\left(x|\gamma_v(r)\right)_o}\:e^{(h-\alpha) r}\:dr\:d\theta_o(v) \\
&=:& I_1(x) \:+\:I_2(x)\:.
\end{eqnarray*}
We first estimate $I_1(x)$. In this case, $r \in \left(0,\left(x|\gamma_v(\infty)\right)_o\right)$ and hence by (\ref{energy_estimates_eq1}),
\begin{equation*}
r-2\delta \le \left(x|\gamma_v(r)\right)_o \le r\:,
\end{equation*}
where the upper bound is always true by the triangle inequality. Plugging the above in the definition of $I_1(x)$, we get
\begin{eqnarray*}
I_1(x) & \asymp & \int_{T^1_o\X}\int_0^{\left(x|\gamma_v(\infty)\right)_o} e^{(h+t) r}\:dr\:d\theta_o(v) \\
& \lesssim & \int_{T^1_o\X} e^{(h+t)\left(x|\gamma_v(\infty)\right)_o }\:d\theta_o(v)\:.
\end{eqnarray*}
Now as the visibility measure $\lambda_o$ on $\partial \X$ is the pushforward of $\theta_o$ on $T^1_o\X$ under the radial projection, we have that
\begin{equation*}
\int_{T^1_o\X} e^{(h+t)\left(x|\gamma_v(\infty)\right)_o }\:d\theta_o(v) = \int_{\partial \X} e^{(h+t)\left(x|\xi\right)_o }\:d\lambda_o(\xi)\:,
\end{equation*}
and hence,
\begin{equation} \label{energy_estimates_eq2}
I_1(x) \lesssim  \int_{\partial \X} e^{(h+t)\left(x|\xi\right)_o }\:d\lambda_o(\xi)\:.
\end{equation}
Similarly, for $I_2(x)$, as $r \in \left[\left(x|\gamma_v(\infty)\right)_o,\infty\right)$, it follows from (\ref{energy_estimates_eq1}) that
\begin{equation*}
\left(x|\gamma_v(\infty)\right)_o -2\delta \le \left(x|\gamma_v(r)\right)_o \le \left(x|\gamma_v(\infty)\right)_o\:,
\end{equation*}
where the upper bound is always true by the monotonicity of Gromov products along geodesic rays. Plugging the above in the definition of $I_2(x)$ and by Fubini-Tonelli, we have
\begin{eqnarray*}
I_2(x) &\asymp & \int_{T^1_o\X} e^{(\alpha +t)\left(x|\gamma_v(\infty)\right)_o} \left(\int_{\left(x|\gamma_v(\infty)\right)_o}^\infty e^{(h-\alpha) r}\:dr \right)\:d\theta_o(v) \\
& \asymp & \int_{T^1_o\X} e^{(h +t)\left(x|\gamma_v(\infty)\right)_o} \:d\theta_o(v)\:.
\end{eqnarray*}
Thus again,
\begin{equation} \label{energy_estimates_eq3}
I_2(x) \lesssim \int_{\partial \X} e^{(h+t)\left(x|\xi\right)_o }\:d\lambda_o(\xi)\:.
\end{equation}
Combining (\ref{energy_estimates_eq2}) and (\ref{energy_estimates_eq3}), we get
\begin{equation} \label{energy_estimates_eq4}
\int_{\X} e^{(\alpha +t)(x|y)_o}\:e^{-\alpha d(o,y)}\:dvol(y) \lesssim \int_{\partial \X} e^{(h+t)\left(x|\xi\right)_o }\:d\lambda_o(\xi)\:.
\end{equation}
Now recalling,
\begin{equation*}
(x|\xi)_o = \frac{1}{2}\left[d(o,x)-B_\xi(x)\right]\:,
\end{equation*}
we note that
\begin{equation*}
\int_{\partial \X} e^{(h+t)\left(x|\xi\right)_o }\:d\lambda_o(\xi) = e^{\left(\frac{h+t}{2}\right)d(o,x)} \int_{\partial \X} e^{-\left(\frac{h+t}{2}\right)B_\xi(x)}\:d\lambda_o(\xi)\:.
\end{equation*}
Thus by the formula for the spherical functions (\ref{spherical_fn}), we identify that
\begin{equation*}
\int_{\partial \X} e^{(h+t)\left(x|\xi\right)_o }\:d\lambda_o(\xi) = e^{\left(\frac{h+t}{2}\right)d(o,x)}\varphi_{i\frac{t}{2}}(x)\:,
\end{equation*}
which upon plugging in the estimates (\ref{spherical_fn_estimate}) yields
\begin{equation} \label{energy_estimates_eq5}
\int_{\partial \X} e^{(h+t)\left(x|\xi\right)_o }\:d\lambda_o(\xi) \asymp e^{td(o,x)}\:.
\end{equation}
Then plugging (\ref{energy_estimates_eq5}) in (\ref{energy_estimates_eq4}) yields the result.
\end{proof}

Next, for $\psi$, a non-negative measurable function on $\X$, we introduce the notation,
\begin{equation} \label{related_harmonic_fn}
H_\psi(x):= e^{-hd(o,x)}\int_\X e^{2h(x|y)_o}\:e^{-hd(o,y)}\:\psi(y)\:dvol(y)\:,\:\:x\in \X\:,
\end{equation}
whenever the integral exists. Our next lemma, studies the tangential boundary behaviour of $H_\psi$:
\begin{lemma} \label{prop2}
Let $\psi$ be a non-negative measurable function on $\X$ such that $\psi\:dvol$ is a Radon measure. If $\xi \in \partial \X$ is such that each of the integrals 
\begin{equation*}
\int_{\Gamma_{\alpha,\tau}(\xi)} e^{\tau \beta_1(x|\xi)_o}\:e^{-\beta_1 d(o,x)}\:\psi(x)^p\:dvol(x)\:\: \text{ and } \int_{\X \setminus \Gamma_{\alpha,\tau}(\xi)} e^{\tau \beta_2(x|\xi)_o}\:e^{-\beta_2 d(o,x)}\:\psi(x)^p\:dvol(x)
\end{equation*}
is finite for some $p>1$ and some $\beta_1,\beta_2 \in (0,h),\:\alpha>1,\:\tau>1$, then
\begin{equation*}
\lim_{\substack{x \to \xi \\ x \in \Gamma_{\alpha,\tau}(\xi)}} H_\psi(x)=0\:.
\end{equation*}
\end{lemma}
\begin{proof}
For $R>0$, let $A_R:= \X \setminus B(o,R)$. Then decomposing the integral in (\ref{related_harmonic_fn}) into $B(o,R)$ and $A_R$ and then proceeding as in the proof of Lemma \ref{non-polar_lemma} for the geodesic ball $B(o,R)$, we get that
\begin{equation*}
H_\psi(x) \le C_R\:e^{-hd(o,x)}\:+\:e^{-hd(o,x)} \int_{A_R} e^{2h(x|y)_o}\:e^{-hd(o,y)}\:\psi(y)\:dvol(y)\:.
\end{equation*}
Now by H\"older's inequality, we have
\begin{eqnarray*}
&&\int_{A_R \cap \Gamma_{\alpha,\tau}(\xi)} e^{2h(x|y)_o}\:e^{-hd(o,y)}\:\psi(y)\:dvol(y) \\
&\le & \left(\int_{A_R} e^{\left(2h-\frac{\beta_1}{p}\right)q(x|y)_o} \: e^{-\left(h-\frac{\beta_1}{p}\right)qd(o,y)}\:dvol(y)\right)^{\frac{1}{q}} \\
&& \times \left(\int_{A_R \cap \Gamma_{\alpha,\tau}(\xi)} e^{\beta_1(x|y)_o}\:e^{-\beta_1d(o,y)}\:\psi(y)^p\:dvol(y)\right)^{\frac{1}{p}}\:.
\end{eqnarray*}
Now by Lemma \ref{energy_estimates}, we have
\begin{equation*}
\left(\int_{A_R} e^{\left(2h-\frac{\beta_1}{p}\right)q(x|y)_o} \: e^{-\left(h-\frac{\beta_1}{p}\right)qd(o,y)}\:dvol(y)\right)^{\frac{1}{q}} \lesssim e^{hd(o,x)}\:,
\end{equation*}
and hence,
\begin{eqnarray*}
&&\int_{A_R \cap \Gamma_{\alpha,\tau}(\xi)} e^{2h(x|y)_o}\:e^{-hd(o,y)}\:\psi(y)\:dvol(y) \\
& \lesssim &  e^{hd(o,x)} \left(\int_{A_R \cap \Gamma_{\alpha,\tau}(\xi)} e^{\beta_1(x|y)_o}\:e^{-\beta_1d(o,y)}\:\psi(y)^p\:dvol(y)\right)^{\frac{1}{p}}\:.
\end{eqnarray*}
A similar estimate also holds for the integral over $A_R \setminus \Gamma_{\alpha,\tau}(\xi)$  and with $\beta_1$ replaced by $\beta_2$. 

Now by (\ref{four_pt_condition'}), we have for any $x,y \in \X$,
\begin{equation*}
e^{-(y|\xi)_o} \le e^{2\delta}\left(e^{-(x|y)_o}\:+\:e^{-(x|\xi)_o}\right)\:.
\end{equation*}
Now if $x \in \Gamma_{\alpha,\tau}(\xi)$, we have
\begin{equation*}
e^{-(x|\xi)_o} < \alpha^{1/\tau} e^{-\frac{d(o,x)}{\tau}}\:,
\end{equation*}
and furthermore, by the triangle inequality, as $(x|y)_o \le d(o,x)$, we get that
\begin{equation*}
e^{-(x|\xi)_o} < \alpha^{1/\tau} e^{-\frac{(x|y)_o}{\tau}}\:.
\end{equation*}
Therefore,
\begin{equation*}
e^{-(y|\xi)_o} < e^{2\delta}\left(e^{-(x|y)_o}\:+\:\alpha^{1/\tau} e^{-\frac{(x|y)_o}{\tau}}\right)\:.
\end{equation*}
Now as $\tau>1$, we have
\begin{equation*}
e^{-(y|\xi)_o} \lesssim e^{-\frac{(x|y)_o}{\tau}}\:,
\end{equation*}
and thus
\begin{equation*}
e^{(x|y)_o} \lesssim e^{\tau(y|\xi)_o}\:,
\end{equation*}
which yields, for $x \in \Gamma_{\alpha,\tau}(\xi)$,
\begin{eqnarray} \label{prop2_pf_eq}
H^p_\psi(x) & \le & 3^{p-1}\:C^p_R\:e^{-hp\:d(o,x)}\:+\:C_2\int_{A_R \cap \Gamma_{\alpha,\tau}(\xi)} e^{\tau \beta_1 (y|\xi)_o}\:e^{-\beta_1d(o,y)}\:\psi(y)^p\:dvol(y) \nonumber\\
&& +\: C_3\int_{A_R \setminus \Gamma_{\alpha,\tau}(\xi)} e^{\tau \beta_2 (y|\xi)_o}\:e^{-\beta_2d(o,y)}\:\psi(y)^p\:dvol(y)\:,
\end{eqnarray}
for some constants $C_2,C_3>0$. We note that, by our hypothesis,
\begin{eqnarray*}
&&\lim_{R \to \infty} \int_{A_R \cap \Gamma_{\alpha,\tau}(\xi)} e^{\tau \beta_1 (y|\xi)_o}\:e^{-\beta_1d(o,y)}\:\psi(y)^p\:dvol(y) \\
=&& \lim_{R \to \infty} \int_{A_R \setminus \Gamma_{\alpha,\tau}(\xi)} e^{\tau \beta_2 (y|\xi)_o}\:e^{-\beta_2d(o,y)}\:\psi(y)^p\:dvol(y) \\
=&&0\:.
\end{eqnarray*}
Thus given $\varepsilon>0$, there exists $R>0$ such that
\begin{eqnarray*}
&&C_2\int_{A_R \cap \Gamma_{\alpha,\tau}(\xi)} e^{\tau \beta_1 (y|\xi)_o}\:e^{-\beta_1d(o,y)}\:\psi(y)^p\:dvol(y) \\
&&+ C_3\int_{A_R \setminus \Gamma_{\alpha,\tau}(\xi)} e^{\tau \beta_2 (y|\xi)_o}\:e^{-\beta_2d(o,y)}\:\psi(y)^p\:dvol(y) \\
&<& \varepsilon\:.
\end{eqnarray*}
For this $R$, we thus have by (\ref{prop2_pf_eq}),
\begin{equation*}
\limsup_{\substack{x \to \xi \\ x \in \Gamma_{\alpha,\tau}(\xi)}} H^p_\psi(x) \le \varepsilon\:.
\end{equation*}
This completes the proof.
\end{proof}

Our next result addresses the size of the sets on $\partial \X$, where the integrals in Lemma \ref{prop2} fail to be finite:
\begin{lemma} \label{prop3}
Let $\mu$ be a Radon measure on $\X$ satisfying 
\begin{equation} \label{beta-hypothesis}
\int_\X e^{-\beta\:d(o,x)}\:d\mu(x) <\infty\:,
\end{equation}
for some $\beta \in (0,h)$. Let $\tau \in (1,h/\beta)$ and $\alpha>1$. Then
\begin{itemize}
\item[(i)] for any $\beta_1 \in (0,\beta)$,
\begin{equation*}
\mathcal{H}^{\beta \tau}\left(\left\{\xi \in \partial \X : \int_{\Gamma_{\alpha,\tau}(\xi)} e^{\tau \beta_1(x|\xi)_o}\:e^{-\beta_1 d(o,x)}\:d\mu(x)=\infty \right\}\right)=0\:;
\end{equation*}
\item[(ii)] for any $\beta_2 \in (\beta,h]$,
\begin{equation*}
\mathcal{H}^{\beta \tau}\left(\left\{\xi \in \partial \X : \int_{\X \setminus \Gamma_{\alpha,\tau}(\xi)} e^{\tau \beta_2(x|\xi)_o}\:e^{-\beta_2 d(o,x)} \:d\mu(x)=\infty \right\}\right)=0\:.
\end{equation*}
\end{itemize}
\end{lemma}
\begin{proof}
To prove part (i), for $\beta_1 \in (0,\beta)$, we consider the function $F$ defined on $\partial \X$ by,
\begin{equation*}
F(\xi):=\int_{\Gamma_{\alpha,\tau}(\xi)} e^{\tau \beta_1(x|\xi)_o}\:e^{-\beta_1 d(o,x)}\:d\mu(x)\:,\:\:\xi \in \partial \X.
\end{equation*}
Again, we recall that $\nu$ only defines a quasi-metric on $\partial \X$. Then as in the proof of part (ii) of Theorem \ref{exceptional_sharp_thm}, we again resort to the visual parameters admissible by the `asymptotic upper curvature bound of $\X$'. Next, incorporating Frostmann's lemma (Lemma \ref{frostmann_lemma}) and the inner regularity of the Hausdorff outer measures (for that visual parameter) and then utilizing the power scaling and bi-Lipschitz invariance of the Hausdorff outer measures, it suffices to prove that $F(\xi)<\infty$ for $\omega$-a.e. $\xi \in \partial \X$, where $\omega$ is any positive measure on $\partial \X$ satisfying,
\begin{equation} \label{Frostmann_condition}
\omega\left(\B(\eta,r)\right) \le Cr^{\beta \tau}\:,\:\:\forall \eta \in \partial \X,\:r>0\:,
\end{equation} 
for some constant $C>0$.

Next, we define, the $\tau$-tangential analogue of Definition \ref{non-tangential-shadow},
\begin{equation*}
\tilde{\Gamma}_{\alpha,\tau}(x):= \left\{\xi \in \partial \X : x \in \Gamma_{\alpha,\tau}(\xi)\right\}\:,\:\:x\in \X\:.
\end{equation*}
Now by Fubini-Tonelli's theorem,
\begin{equation*}
\int_{\partial \X} F(\xi)\:d\omega(\xi)=\int_\X e^{-\beta_1 d(o,x)} \left(\int_{\tilde{\Gamma}_{\alpha,\tau}(x)} e^{\tau \beta_1(x|\xi)_o}\:d\omega(\xi)\right) d\mu(x)\:.
\end{equation*}
We fix $x \in \X$ and focus on estimating the integral,
\begin{equation*}
\int_{\tilde{\Gamma}_{\alpha,\tau}(x)} e^{\tau \beta_1(x|\xi)_o}\:d\omega(\xi)\:.
\end{equation*}
For $k \in \N \cup \{0\}$, we define the sets 
\begin{eqnarray*}
U_k&:=& \left\{\xi \in \partial \X : e^{-(x|\xi)_o} < \left(\alpha\:e^{-d(o,x)}\right)^{\frac{1}{\tau}}\:2^{-k}\right\}\:, \\
V_k&:=& U_k \setminus U_{k+1}\:.
\end{eqnarray*}
Clearly, we have $\tilde{\Gamma}_{\alpha,\tau}(x) \subset \displaystyle\bigcup_{k=0}^\infty V_k$ and thus,
\begin{eqnarray*}
\int_{\tilde{\Gamma}_{\alpha,\tau}(x)} e^{\tau \beta_1(x|\xi)_o}\:d\omega(\xi) &\le& \sum_{k=0}^\infty \int_{V_k} e^{\tau \beta_1(x|\xi)_o}\:d\omega(\xi) \\
& \lesssim & e^{\beta_1\:d(o,x)}\: \sum_{k=0}^\infty \left(2^{k}\right)^{\tau \beta_1}\:\omega\left(U_k\right)\:.
\end{eqnarray*}
Now let $\eta_x \in \partial \X$ be the projection of $x$ from $o$ to $\partial \X$. Then as Gromov products are monotonically increasing along geodesic rays, as in the proof of Lemma \ref{shadow_lemma1}, we get that
\begin{equation*}
U_k \subset \B\left(\eta_x,\left(\alpha\:e^{-d(o,x)}\right)^{\frac{1}{\tau}}\:2^{-k}\right)\:,
\end{equation*}
which along with (\ref{Frostmann_condition}) yields the estimate,
\begin{equation*}
\omega\left(U_k\right) \lesssim e^{-\beta d(o,x)}\:\left(2^{-k}\right)^{\beta \tau}\:.
\end{equation*}
Therefore,
\begin{equation*}
\int_{\tilde{\Gamma}_{\alpha,\tau}(x)} e^{\tau \beta_1(x|\xi)_o}\:d\omega(\xi) \lesssim e^{-(\beta-\beta_1)d(o,x)} \sum_{k=0}^\infty \left(2^{-(\beta-\beta_1)\tau}\right)^k\:.
\end{equation*}
Since, $\beta_1<\beta$, the above series is convergent and hence by (\ref{beta-hypothesis}),
\begin{equation*}
\int_{\partial \X} F(\xi)\:d\omega(\xi) \lesssim \int_X e^{-\beta\:d(o,x)}\:d\mu(x) < \infty\:.
\end{equation*}
Therefore, $F(\xi)<\infty$ for $\omega$-a.e. $\xi \in \partial \X$. This completes the proof of part (i). 

\medskip

To prove part (ii), we let $\beta_2 \in (\beta,h]$ and $\omega$ be as in (\ref{Frostmann_condition}). For $R>0$, we set
\begin{equation*}
F_R(\xi):=\int_{A_R \setminus \Gamma_{\alpha,\tau}(\xi)} e^{\tau \beta_2(x|\xi)_o}\:e^{-\beta_2 d(o,x)} \:d\mu(x)\:,\:\:\xi \in \partial \X\:,
\end{equation*}
where we recall that $A_R = \X \setminus B(o,R)$. It is clear that the integral in part (ii) of Lemma \ref{prop3} is finite if and only if $F_R$ is finite for some $R>0$. Again by Fubini-Tonelli's theorem,
\begin{equation*}
\int_{\partial \X} F_R(\xi)\:d\omega(\xi)=\int_{A_R} e^{-\beta_2 d(o,x)} \left(\int_{\partial \X \setminus \tilde{\Gamma}_{\alpha,\tau}(x)} e^{\tau \beta_2(x|\xi)_o}\:d\omega(\xi)\right) d\mu(x)\:.
\end{equation*}
Fixing $x \in A_R$, we now focus on estimating the integral,
\begin{equation*}
\int_{\partial \X \setminus \tilde{\Gamma}_{\alpha,\tau}(x)} e^{\tau \beta_2(x|\xi)_o}\:d\omega(\xi)\:.
\end{equation*}
Let $\xi \in \partial \X \setminus \tilde{\Gamma}_{\alpha,\tau}(x)$. Thus
\begin{equation*}
\left(\alpha \:e^{-d(o,x)}\right)^\frac{1}{\tau} \le e^{-(x|\xi)_o}\:.
\end{equation*}
Next let $\eta_x \in \partial \X$ be the projection of $x$ from $o$ to $\partial \X$. Then by (\ref{four_pt_condition'}), we have
\begin{equation*}
e^{-(x|\xi)_o} \le e^{2\delta} \left(e^{-(\eta_x|\xi)_o}\:+\:e^{-(\eta_x|x)_o}\right)= e^{2\delta} \left(e^{-(\eta_x|\xi)_o}\:+\:e^{-d(o,x)}\right)\:. 
\end{equation*} 
Hence,
\begin{equation*}
\left(\alpha \:e^{-d(o,x)}\right)^\frac{1}{\tau} \le e^{2\delta} \left(e^{-(\eta_x|\xi)_o}\:+\:e^{-d(o,x)}\right)\:,
\end{equation*}
which yields,
\begin{equation*}
e^{-\frac{1}{\tau}d(o,x)}\left(\alpha^\frac{1}{\tau} e^{-2\delta}\:-\:e^{-\left(1-\frac{1}{\tau}\right)d(o,x)}\right) \le e^{-(\eta_x|\xi)_o}\:.
\end{equation*}
Now as $\tau>1$, there exists $C'>0$ and $R>0$ such that
\begin{equation*}
\partial \X \setminus \tilde{\Gamma}_{\alpha,\tau}(x) \subset \mathcal{E}(x):=\left\{\xi \in \partial \X : e^{-(\eta_x|\xi)_o} \ge C'e^{-\frac{1}{\tau}d(o,x)}\right\}\:,
\end{equation*}
for all $x \in A_R$. Thus,
\begin{equation*}
\int_{\partial \X \setminus \tilde{\Gamma}_{\alpha,\tau}(x)} e^{\tau \beta_2(x|\xi)_o}\:d\omega(\xi) \le \int_{\mathcal{E}(x)} e^{\tau \beta_2(\eta_x|\xi)_o}\:d\omega(\xi)\:.
\end{equation*}
There exists a unique $N_x \in \N$, such that
\begin{equation} \label{defn_Nx}
e^{-N_x}< C' e^{-\frac{1}{\tau}d(o,x)} \le e^{-N_x+1}\:.
\end{equation}
Then for $k=0,1,\dots,N_x$, let
\begin{eqnarray*}
\mathcal{U}_k &:=& \left\{\xi \in \partial \X : e^{-(\eta_x|\xi)_o} \le e^{-k+1} \right\}\:,\\
\mathcal{V}_k &:=& \left\{\xi \in \partial \X : e^{-k} < e^{-(\eta_x|\xi)_o} \le e^{-k+1} \right\}\:.
\end{eqnarray*}
Clearly as $\mathcal{E}(x) \subset \displaystyle\bigcup_{k=0}^{N_x} \mathcal{V}_k$, we have
\begin{equation*}
\int_{\mathcal{E}(x)} e^{\tau \beta_2(\eta_x|\xi)_o}\:d\omega(\xi)  \le  \sum_{k=0}^{N_x} \int_{\mathcal{V}_k} e^{\tau \beta_2(\eta_x|\xi)_o}\:d\omega(\xi) \le \sum_{k=0}^{N_x} e^{\tau \beta_2 k}\:\omega\left(\mathcal{U}_k\right)\:.
\end{equation*}
Then by (\ref{Frostmann_condition}) and the assumption that $\beta_2>\beta$, it follows that
\begin{equation*}
\int_{\mathcal{E}(x)} e^{\tau \beta_2(\eta_x|\xi)_o}\:d\omega(\xi)  \lesssim  \sum_{k=0}^{N_x} \left(e^{\tau (\beta_2-\beta)}\right)^k \lesssim e^{\tau(\beta_2-\beta)(N_x-1)}\:.
\end{equation*}
Now by the definition (\ref{defn_Nx}), we get 
\begin{equation*}
\int_{\mathcal{E}(x)} e^{\tau \beta_2(\eta_x|\xi)_o}\:d\omega(\xi)  \lesssim e^{(\beta_2-\beta)d(o,x)}\:,
\end{equation*}
which in turn by (\ref{beta-hypothesis}) yields,
\begin{equation*}
\int_{\partial \X} F_R(\xi)\:d\omega(\xi) \lesssim \int_{A_R} e^{-\beta\:d(o,x)}\:d\mu(x) <\infty\:.
\end{equation*}
Hence, $F_R(\xi)<\infty$, for $\omega$-a.e. $\xi \in \partial \X$. This completes the proof of Lemma \ref{prop3}.
\end{proof}

As consequences, we obtain the following results:
\begin{corollary} \label{Cor2}  
Let $\psi$ be a non-negative measurable function on $\X$ such that $\psi\:dvol$ is a Radon measure. If $\psi$ satisfies (\ref{Lp_condition'}) for some $\beta \in (0,h)$ and some $p>1$, then for each $\tau \in (1,h/\beta)$, there exists a set $E \subset \partial \X$ with $\mathcal{H}^{\beta \tau}(E)=0$, such that $H_\psi$ has tangential limit $0$ of degree $\tau$ at all $\xi \in \partial \X \setminus E.$
\end{corollary}
\begin{proof}
The result follows at once from Lemmata \ref{prop2} and \ref{prop3} by taking the measure, $d\mu:=\psi^p\:dvol$.
\end{proof}

\begin{corollary} \label{Cor3}
Let $\mu$ and $\beta$ be as in Lemma \ref{prop3}. The for each $\alpha>1,\:\tau \in (1,h/\beta)$,
\begin{equation*}
\mathcal{H}^{\beta \tau}\left(\left\{\xi \in \partial \X:\mu\left(\Gamma_{\alpha,\tau}(\xi)\right)=\infty\right\}\right)=0\:.
\end{equation*}
\end{corollary}
\begin{proof}
As for any $\beta_1 \in (0,\beta)$,
\begin{equation*}
\int_{\Gamma_{\alpha,\tau}(\xi)} e^{\tau \beta_1(x|\xi)_o}\:e^{-\beta_1 d(o,x)}\:d\mu(x) > \frac{\mu\left(\Gamma_{\alpha,\tau}(\xi)\right)}{\alpha^{\beta_1}}\:,
\end{equation*}
the result follows from part (i) of Lemma \ref{prop3}.
\end{proof}

We now present the proof of Theorem \ref{tangential_limit_green_thm}.
\begin{proof}[Proof of Theorem \ref{tangential_limit_green_thm}]
Let $\beta \in (0,h),\:p>n/2,\:\tau \in (1,h/\beta)$. We start off by decomposing the Green potential into polar and non-polar parts:
\begin{equation*}
G[\psi](x)=u_1(x)\:+\:u_2(x)\:,\:\:x \in \X\:,
\end{equation*}
where
\begin{eqnarray*}
&&u_1(x):= \int_{B(x,1)} G(x,y)\:\psi(y)\:dvol(y)\:,\\
&&u_2(x):= \int_{\X \setminus B(x,1)} G(x,y)\:\psi(y)\:dvol(y)\:.
\end{eqnarray*} 

We first focus on $u_1$. Proceeding as in the proof of Lemma \ref{polar_lemma}, we get
\begin{equation*}
u_1(x) \lesssim \left(\int_{B(x,1)} \psi(y)^p\:dvol(y)\right)^{\frac{1}{p}}\:,
\end{equation*}
and that for any $\xi \in \partial \X,\alpha>1$ and $R>0$, if $x \in \Gamma_{\alpha,\tau}(\xi) \cap A_{R+1}$, then $B(x,1) \subset \Gamma_{e^{\tau+1}\alpha,\tau}(\xi) \cap A_R$. Thus, we have for $x \in \Gamma_{\alpha,\tau}(\xi) \cap A_{R+1}$,
\begin{equation} \label{green_thm2_eq1}
u_1(x) \lesssim \left(\int_{\Gamma_{e^{\tau+1}\alpha,\tau}(\xi) \cap A_R} \psi(y)^p\:dvol(y)\right)^{\frac{1}{p}}\:.
\end{equation}
Next setting $d\mu:=\psi^p\:dvol$, by Corollary \ref{Cor3}, there exists a set $E_1 \subset \partial \X$ with $\mathcal{H}^{\beta \tau}(E_1)=0$ such that for all $\xi \in \partial \X \setminus E_1$,
\begin{equation*}
\mu\left(\Gamma_{e^{\tau+1}\alpha,\tau}(\xi)\right)=\int_{\Gamma_{e^{\tau+1}\alpha,\tau}(\xi)} \psi(y)^p\:dvol(y) <\infty\:.
\end{equation*}
Thus for all $\xi \in \partial \X \setminus E_1$, by (\ref{green_thm2_eq1}), we get
\begin{equation*} 
\limsup_{\substack{x \to \xi \\ x \in \Gamma_{\alpha,\tau}(\xi)}} u_1(x) \lesssim \lim_{R \to \infty} \left(\int_{\Gamma_{e^{\tau+1}\alpha,\tau}(\xi) \cap A_R} \psi(y)^p\:dvol(y)\right)^{\frac{1}{p}}=0\:.
\end{equation*}
We are now left with $u_2$. By the pointwise decay of the Green function given by (\ref{green_estimate}), away from the pole, we have 
\begin{equation*}
u_2(x) \lesssim \int_{\X} e^{-hd(x,y)}\:\psi(y)\:dvol(y)\:.
\end{equation*}
Now as,
\begin{equation*}
d(x,y)=d(o,x)\:+\:d(o,y)\:-\:2(x|y)_o\:,
\end{equation*}
we identify by (\ref{related_harmonic_fn}) that
\begin{eqnarray*}
u_2(x) &\lesssim & e^{-hd(o,x)}\int_{\X} e^{2h(x|y)_o}\:e^{-hd(o,y)}\:\psi(y)\:dvol(y) \\
&=& H_\psi(x)\:.
\end{eqnarray*}
Then by Corollary \ref{Cor2}, there exists a set $E_2 \subset \partial \X$ with $\mathcal{H}^{\beta \tau}(E_2)=0$, such that $u_2$ has tangential limit $0$ of degree $\tau$ at all $\xi \in \partial \X \setminus E_2$. Thus Theorem \ref{tangential_limit_green_thm} follows by taking $E:=E_1 \cup E_2$.
\end{proof}

\section{Non-tangential, tangential limits and exceptional sets of superharmonic functions} \label{sec7}
In this section, we complete the proofs of Theorems \ref{non-tangential_superharmonic_thm} and \ref{tangential_limit_superharmonic_thm}.

\begin{proof}[Proof of Theorem \ref{non-tangential_superharmonic_thm}]
By the Riesz decomposition theorem (Lemma \ref{riesz_decomposition_lemma}) and the Martin representation formula (\ref{martin_rep}), there exists a unique Radon measure $\mu$ on $\partial \X$ (the boundary measure of the greatest harmonic minorant of $f$) such that
\begin{equation*}
f=\PP[\mu]\:+\:G[\psi]\:,\:\:\text{ on } \X.
\end{equation*}
Let $\omega$ be the absolutely continuous component of $\mu$ with respect to $\lambda_o$. Then by Theorem \ref{Fatou_thm}, there exists a set of boundary points, say $E_1$, with $\lambda_o(E_1)=1$ such that the non-tangential limits of $\PP[\mu]$ exist at every $\xi \in E_1$ and equal the Radon-Nikodym derivative $\frac{d\omega}{d\lambda_o}$. 

On the other hand, by Theorem \ref{non-tangential_green_thm}, there exists a set of boundary points, say $E_2$, with $\lambda_o(E_2)=1$ such that the non-tangential limits of $G[\psi]$ exist at every $\xi \in E_2$ and equal $0$. 

Thus the conclusion follows on $E_1 \cap E_2$.
\end{proof}

\begin{proof}[Proof of Theorem \ref{tangential_limit_superharmonic_thm}]
By the Riesz decomposition theorem (Lemma \ref{riesz_decomposition_lemma}), there exists a unique non-negative harmonic function $F_f$ (the greatest harmonic minorant of $f$) such that 
\begin{equation*}
f=F_f\:+\:G[\psi]\:,\:\:\text{ on } \X.
\end{equation*}
Moreover, $F_f$ is explicitly given by (\ref{form_harmonic_minorant}):
\begin{equation*} 
F_f(x) = \displaystyle\lim_{r \to +\infty} \int_{T^1_{x}X} f\left(\gamma_{x,v}(r)\right)\: d\theta_{x}(v) \:,\:\: x \in \X\:.
\end{equation*}
Thus by the hypothesis (\ref{spherical_average_decay}), we have that $F_f(x_0)=0$. Then by the maximum principle, $F_f \equiv 0$ and hence, $f=G[\psi]$. The result now follows from Theorem \ref{tangential_limit_green_thm}.
\end{proof}

\section{Some counter-examples} \label{sec8}
In this section, we prove Theorems \ref{first_counter-example} and \ref{second_counter-example}. But before that, we prove the following result relating Riemannian angle and the Gromov product:

\begin{lemma} \label{angle_gromov_relation}
Let $\X$ be a Hadamard manifold of dimension $\ge 2$, with sectional curvature bounds $-b^2 \le K_\X \le 0$. Fix $x \in \X$. Let $y \in \X$ and $\xi \in \partial \X$ such that $x,y$ and $\xi$ are not collinear. Then the Riemannian angle between $y$ and $\xi$ subtended at $x$, say $\theta$, satisfies 
\begin{equation*}
e^{-2b{(y|\xi)}_x} - e^{-2bd(x,y)} \le \sin^2 \frac{\theta}{2} \:.
\end{equation*}
\end{lemma}
\begin{proof}
Let $\gamma$ denote the geodesic ray starting from $x$ and hitting $\partial \X$ at $\xi$. Now for any $t\in (0,\infty)$, we consider the geodesic triangle $\Delta(x,y,\gamma(t))$. Then let $\theta^t_b$ be the angle corresponding to $\theta$ in the comparison triangle $\overline{\Delta}(x,y,\gamma(t))$ in $\mathbb H^2(-b^2)$. Now the curvature pinching condition yields by an application of (\ref{finite_angle_comparison}) that 
\begin{equation*}
\sin\frac{\theta^t_b}{2}\le \sin\frac{\theta}{2} \quad \text{for all} \quad t\in (0,\infty)\:.
\end{equation*} 
Then by the hyperbolic law of cosines,
\begin{equation*}
\sin^2\frac{\theta^t_b}{2}=\frac{\cosh bd(y,\gamma(t))-\cosh b(d(y,x)-d(\gamma(t),x))}{2\sinh bd(y,x)\sinh bd(\gamma(t),x)}
\end{equation*}
Now
\begin{align*}
&\lim_{t\to\infty} \frac{\cosh bd(y,\gamma(t))}{2\sinh bd(y,x)\sinh bd(\gamma(t),x)}\\
=&\lim_{t\to\infty} \frac{e^{bd(y,\gamma(t))}+e^{-bd(y,\gamma(t))}}{(e^{bd(y,x)}-e^{-bd(y,x)})(e^{bd(x,\gamma(t))}-e^{-bd(x,\gamma(t))})}\\
=&\frac{e^{-2b(y|\xi)_x}}{1-e^{-2bd(x,y)}}\:,
\end{align*}
and
\begin{align*}
&\lim_{t\to\infty} \frac{\cosh b(d(y,x)-d(\gamma(t),x))}{2\sinh bd(y,x)\sinh bd(\gamma(t),x)}\\
=&\lim_{t\to\infty} \frac{e^{b(d(y,x)-d(\gamma(t),x))}+e^{-b(d(y,x)-d(\gamma(t),x))}}{(e^{bd(y,x)}-e^{-bd(y,x)})(e^{bd(x,\gamma(t))}-e^{-bd(x,\gamma(t))})}\\
=&\frac{e^{-2bd(x,y)}}{1-e^{-2bd(x,y)}}\:.
\end{align*}
Hence,
\begin{equation*}
e^{-2b{(y|\xi)}_x} - e^{-2bd(x,y)} \le\frac{e^{-2b{(y|\xi)}_x} - e^{-2bd(x,y)}}{1-e^{-2bd(x,y)}} \le \sin^2 \frac{\theta}{2} \:.
\end{equation*}
\end{proof} 

\begin{proof}[Proof of Theorem \ref{first_counter-example}]
Given $\alpha>1$, we first construct a countable set $\mathcal{A} \subset \X$ such that given any point $\xi \in \partial \X$, there exists a sequence $\{x_k\}_{k=1}^\infty \subset \mathcal{A} \cap \Gamma_\alpha(\xi)$ such that $x_k \to \xi$ as $k \to \infty$.

We start by recalling that the sectional curvature of harmonic manifolds are bounded below (see \cite{Besse, RR}), i.e., there exists $b>0$, such that $K_\X \ge -b^2$. Thus along with our non-positivity assumption, we get the pinching condition $-b^2 \le K_\X \le 0$. Now take any sequence $\{t_k\}_{k=1}^\infty$ of positive real numbers increasing to $\infty$. On each geodesic sphere $S(o,t_k)$, let $\mathcal{A}_k$ denote a maximal $r_k$-net with respect to the Riemannian angle metric $\theta_o(\cdot,\cdot)$ where
\begin{equation} \label{counter1_eq1}
r_k:= \left(\alpha^{2b}-1\right)^\frac{1}{2}\:e^{-bt_k}\:.
\end{equation} 
We consider,
\begin{equation*}
\mathcal{A}:= \bigcup_{k=1}^\infty \mathcal{A}_k\:.
\end{equation*}
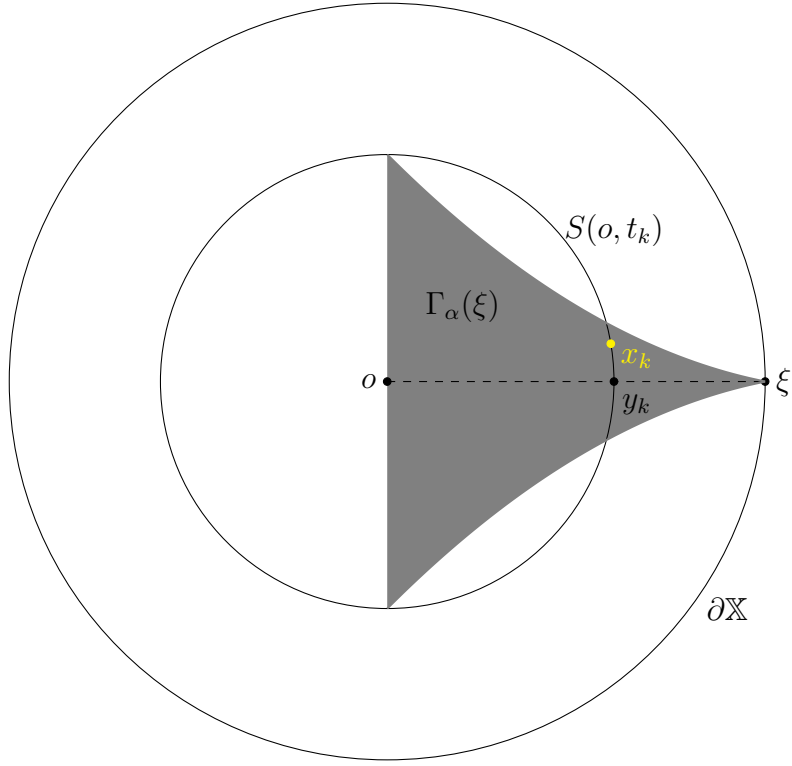
\begin{figure}[h]
\begin{center}
\begin{tikzpicture}
	\fill[gray] (5,8) --plot[domain=5:10, smooth] (\x,{(2*(\x^2)/25)-(9*(\x)/5)+15}) -- plot[domain=10:5, smooth] (\x,{-(2*(\x^2)/25)+(9*(\x)/5)-5}) -- cycle;
	\draw  (5,5) circle [radius=5];
	\draw  (5,5) circle [radius=3];
	\draw [fill] (5,5) circle [radius=.05] node [left] {$o$};
	\draw [dashed] (5,5) -- (10,5);
	\draw [fill] (8,5) circle [radius=.05] node [below] {};
	\pgftext[at={\pgfpoint{8.3cm}{4.7cm}}]{$y_k$};
	\draw [fill] (10,5) circle [radius=.05] node [right] {$\xi$};
	\draw[gray][thick] plot[domain=5:10, smooth] (\x,{(2*(\x^2)/25)-(9*(\x)/5)+15});
	\draw[gray][thick] plot[domain=5:10, smooth] (\x,{-(2*(\x^2)/25)+(9*(\x)/5)-5});
	\draw [yellow][fill] (7.958,5.5) circle [radius=.05] node [below] {};
	\pgftext[at={\pgfpoint{8.3cm}{5.3cm}}]{\textcolor{yellow}{$x_k$}};
	\pgftext[at={\pgfpoint{9.5cm}{2cm}}]{$\partial \X$};
	\pgftext[at={\pgfpoint{8cm}{7cm}}]{$S(o,t_k)$};
	\pgftext[at={\pgfpoint{6cm}{6cm}}]{$\Gamma_\alpha(\xi)$};
	 \end{tikzpicture}
\end{center}
\caption{The sequence $\{x_k\}_{k=1}^\infty$}
\label{figure2}
\end{figure}
Now let $\xi \in \partial \X$ and consider $y_k:=\gamma_\xi(t_k) \in S(o,t_k)$. If $y_k \in \mathcal{A}_k$, then we set $x_k:=y_k$. Otherwise, by construction there exists $x \in \mathcal{A}_k$ such that 
\begin{equation} \label{counter1_eq2}
\theta_o(x,\xi)=\theta_o(x,y_k) < r_k\:.
\end{equation}
We set $x_k:=x$. To show that $x_k \in \Gamma_\alpha(\xi)$, we apply Lemma \ref{angle_gromov_relation}, (\ref{counter1_eq2}), (\ref{counter1_eq1}) and get that
\begin{eqnarray*}
e^{-2b(x_k|\xi)_o} &\le & e^{-2bd(o,x_k)} \:+\: \sin^2\left(\frac{\theta_o(x_k,\xi)}{2}\right) \\
&<& e^{-2bt_k}\:+\: \frac{r^2_k}{4} \\
&=& e^{-2bt_k}\:+\:\left(\frac{\alpha^{2b}-1}{4}\right)e^{-2bt_k}\:.
\end{eqnarray*}
Now as $\alpha>1$, it follows that
\begin{equation*}
e^{-2b(x_k|\xi)_o} < \alpha^{2b}\:e^{-2bt_k}\:,
\end{equation*}
i.e.
\begin{equation*}
e^{-(x_k|\xi)_o} < \alpha\:e^{-t_k}\:,
\end{equation*}
and hence, $x_k \in \Gamma_\alpha(\xi)$. Then as $t_k \uparrow \infty$, clearly, $x_k \to \xi$. This completes the construction of $\mathcal{A}$.

Next, we enumerate the elements of $\mathcal{A}$ (in increasing order of $t_k$) as $\{x_j\}_{j=1}^\infty$ and choose a sequence of positive real numbers $\{\tau_j\}_{j=1}^\infty$ such that
\begin{equation} \label{counter1_eq3}
\sum_{j=1}^\infty \tau_j\: e^{-hd(o,x_j)} < \infty\:.
\end{equation}
We now define the measure $\mu$ on $\X$ by,
\begin{equation*}
\mu:=\sum_{j=1}^\infty \tau_j\:\delta_{x_j}\:,
\end{equation*}
where $\delta_x$ denotes the Dirac point mass at $x$. Then by (\ref{counter1_eq3}) and Lemma \ref{Green_potential_lemma}, $G[\mu]$ is well-defined and thus a positive superharmonic function. Finally, by construction, we note that for every $\xi \in \partial \X$, we have a sequence of poles of $G[\mu]$ converging to $\xi$ within $\Gamma_\alpha(\xi)$. This completes the proof.
\end{proof}

\begin{proof}[Proof of Theorem \ref{second_counter-example}]
Let $\alpha >1$. Now, as in the proof of Theorem \ref{first_counter-example}, let $\{x_j\}_{j=1}^\infty \subset \X$ such that  given any $\xi \in \partial \X$, there is a subsequence $\{x_{j_k}\}_{k=1}^\infty \subset \Gamma_\alpha(\xi)$ with $x_{j_k} \to \xi$ as $k \to \infty$.

Let $p \in [1,n/2]$. For each $j$, we choose $r_j \in (0,1)$ and $\tau_j>0$ such that
\begin{itemize}
\item[(a)] the family of geodesic balls $\{B(x_j,r_j)\}_{j=1}^\infty$ is pairwise disjoint and
\item[(b)] $\displaystyle\sum_{j=1}^\infty e^{-d(o,x_j)}\:\tau^{1/p}_j < \infty\:.$
\end{itemize}
Next, let $\{a_j\}_{j=1}^\infty$ be a sequence of positive real numbers so that $a_j \uparrow \infty$. For each $j$, we choose a non-negative measurable function $\psi_j$ satisfying,
\begin{enumerate}
\item $Supp(\psi_j) \subset B(x_j,r_j)\:,$
\item $\int_{B(x_j,r_j)} \psi^p_j\:dvol < \tau_j\:$ and
\item $\int_\X \psi_j(y)\:G(x_j,y)\:dvol(y) >a_j\:.$
\end{enumerate} 
Clearly, for each $j$, we can find a $\psi_j$ satisfying (1) and (2) above. Now if we cannot find a $\psi_j$ satisfying (1)-(3), then we have
\begin{equation*}
\int_{B(x_j,r_j)} \psi(y)\:G(x_j,y)\:dvol(y) \le a_j\tau^{-1/p}_j\:,
\end{equation*}
for all non-negative measurable functions $\psi$ satisfying $\int_{B(x_j,r_j)}\psi^{p}\:dvol \le 1\:.$ By duality, it then follows that $G_{x_j} \in L^q\left(B(x_j,r_j)\right)$, for $q \ge n/(n-2)$, a contradiction. Hence, we can always make our choice of such $\psi_j$.

We now set, 
\begin{equation*}
\psi:=\sum_{j=1}^\infty \psi_j\:.
\end{equation*}
We first note that by radiality of the volume density, we have a uniform positive constant $C>0$ such that $vol\left(B(x_j,r_j)\right) \le C$. We now estimate,
\begin{eqnarray*}
\int_\X e^{-hd(o,y)}\:\psi(y)\:dvol(y) = \sum_{j=1}^\infty \int_{B(x_j,r_j)} e^{-hd(o,y)}\:\psi_j(y)\:dvol(y)\:.
\end{eqnarray*}
Next, by the triangle inequality and H\"older's inequality,
\begin{eqnarray*}
\sum_{j=1}^\infty \int_{B(x_j,r_j)} e^{-hd(o,y)}\:\psi_j(y)\:dvol(y) &\lesssim &   \sum_{j=1}^\infty e^{-hd(o,x_j)} \int_{B(x_j,r_j)} \psi_j(y)\:dvol(y) \\
&\lesssim &  \sum_{j=1}^\infty e^{-hd(o,x_j)} \left(\int_{B(x_j,r_j)} \psi^p_j\:dvol\right)^{\frac{1}{p}}
\end{eqnarray*}
Thus by the assumption (2) above, we have
\begin{equation*}
\int_\X e^{-hd(o,y)}\:\psi(y)\:dvol(y) \lesssim \sum_{j=1}^\infty e^{-hd(o,x_j)} \tau^{\frac{1}{p}}_j\:.
\end{equation*}
Then as $h \ge n/2 >1$, by the assumption (b) above, we get that
\begin{equation*}
\int_\X e^{-hd(o,y)}\:\psi(y)\:dvol(y) < \infty\:.
\end{equation*}
Hence by Lemma \ref{Green_potential_lemma}, $G[\psi]$ a positive superharmonic function on $\X$ with $\psi$ being the density of its Riesz measure with respect to the volume measure.  We next verify the conditions (\ref{density_counter_condition}) and (\ref{boundary_blowup}).

By the disjointness of the supports of the functions $\psi_j$ and the triangle inequality,
\begin{equation*}
\int_\X e^{-hd(o,y)}\:\psi(y)^p\:dvol(y) \lesssim \sum_{j=1}^\infty e^{-hd(o,x_j)} \int_{B(x_j,r_j)} \psi^p_j\:dvol\:.
\end{equation*}
Then again by the assumption (2) above,
\begin{equation*}
\int_\X e^{-hd(o,y)}\:\psi(y)^p\:dvol(y) \lesssim \sum_{j=1}^\infty e^{-hd(o,x_j)} \tau_j\:.
\end{equation*}
Now by the convergence of the series in (b), there exists $j_0 \in \N$ such that for all $j \ge j_0$,
$$e^{-d(o,x_j)}\:\tau^{1/p}_j \le 1\:.$$
Now as $1 \le p \le n/2 \le h$, we have for all $j \ge j_0$,
\begin{equation*}
e^{-hd(o,x_j)} \tau_j \le \left(e^{-d(o,x_j)}\:\tau^{1/p}_j\right)^p \le e^{-d(o,x_j)}\:\tau^{1/p}_j\:.
\end{equation*}
Thus,
\begin{equation*}
\int_\X e^{-hd(o,y)}\:\psi(y)^p\:dvol(y) \lesssim \sum_{j=1}^\infty e^{-d(o,x_j)} \tau^{1/p}_j <\infty\:,
\end{equation*}
verifying (\ref{density_counter_condition}). Finally, since by (3), we have
\begin{equation*}
G[\psi](x_j) \ge \int_\X \psi_j(y)\:G(x_j,y)\:dvol(y) >a_j\:,
\end{equation*}
by construction of the sequences $\{x_j\}_{j=1}^\infty$ and $\{a_j\}_{j=1}^\infty$, the condition (\ref{boundary_blowup}) is also satisfied.
\end{proof}

\section{Concluding remarks} \label{sec9}
In this section, we make some remarks and pose some new problems:
\begin{enumerate}
\item In Theorems \ref{Fatou_thm}, \ref{exceptional_thm} and \ref{exceptional_sharp_thm}, we study the boundary behaviour of generalized Poisson integrals, corresponding to eigenfunctions of $\Delta$ with eigenvalues $\beta^2-\rho^2$ for $\beta>0$. A natural question now is whether these are all such eigenfunctions? The key to answering this question lies in obtaining a Martin representation formula for these eigenfunctions. In the setting of Damek-Ricci spaces, it is well-known \cite[Theorem 7.11]{DR_JGEA}. Thus a natural attempt would be to generalize this to harmonic manifolds of purely exponential volume growth. 
\item In Theorems \ref{Fatou_thm} and \ref{non-tangential_superharmonic_thm}, we have assumed $\left(\partial \X,\nu,\lambda_o\right)$ to be $h$-doubling. This is indeed true for the known cases of rank one Riemannian symmetric spaces of non-compact type and Damek-Ricci spaces. However, whether this holds true for general harmonic manifolds of purely exponential volume growth, remains to be seen. 
\item Whether, the condition on the decay of spherical averages given by (\ref{spherical_average_decay}) in Theorem \ref{tangential_limit_superharmonic_thm}, can be replaced by a suitable adaptation of \cite{Nagel,Cifuentes} seems to be an interesting problem.
\end{enumerate}

\section*{Acknowledgements}  The author is supported by the Institute Post Doctoral Fellowship of Indian Institute of Technology, Bombay.

\bibliographystyle{amsplain}

\end{document}